\documentclass[twoside,12pt,a4paper]{amsart} 
\usepackage[utf8]{inputenc}
\usepackage[english]{babel}

\usepackage{amsmath}
\usepackage{amsfonts}
\usepackage{amssymb}
\usepackage{color}

\usepackage{amsthm}

\newtheorem{lemma}{Lemma} [section]

\newtheorem{theorem}{Theorem} [section]

\newtheorem{open}{Open Question}

\usepackage{graphics}
\usepackage{epsfig}

\usepackage{url}
\usepackage{hyperref}

\usepackage[a4paper,top=2.5cm,bottom=2.8cm,
left=3.2cm,right=3.2cm]{geometry}
\markleft{\hfill \textsc{Relationships Between Six Excircles} \hfill}
\begin{document}
Sangaku Journal of Mathematics (SJM) \copyright SJM \\
ISSN 2534-9562 \\
Volume 3 (2019), pp.73-90  \\
Received 20 August 2019. Published on-line 30 September 2019 \\ 
web: \url{http://www.sangaku-journal.eu/} \\
\copyright The Author(s) This article is published 
with open access\footnote{This article is distributed under the terms of the Creative Commons Attribution License which permits any use, distribution, and reproduction in any medium, provided the original author(s) and the source are credited.}. \\
\bigskip
\bigskip

\begin{center}
{\Large \textbf{Relationships Between Six Excircles}} \\
\medskip
\bigskip
\textsc{Stanley Rabinowitz} \\
545 Elm St Unit 1,  Milford, New Hampshire 03055, USA \\
e-mail: \href{mailto:stan.rabinowitz@comcast.net}{stan.rabinowitz@comcast.net} \\
web: \url{http://www.StanleyRabinowitz.com/} \\
\end{center}
\bigskip

\textbf{Abstract.} If $P$ is a point inside $\triangle ABC$, then the cevians
through $P$ divide $\triangle ABC$ into smaller triangles of various sizes.
We give theorems about the relationship between the radii of certain
excircles of some of these triangles.



\bigskip
\bigskip
\section{Introduction}

\newcommand{\degrees}{^\circ}

Let $P$ be any point inside a triangle $ABC$. The cevians
through $P$ divide $\triangle ABC$ into six smaller triangles.
In a previous paper \cite{Rabinowitz}, we found relationships between the radii of
the circles inscribed in these triangles.

For example, if $P$ is at the orthocenter $H$, as shown in Figure \ref{fig:incirclesH},
then we found that $r_1r_3r_5=r_2r_4r_6$, where the $r_i$ are radii of the incircles as shown in the figure.

\begin{figure}[h!t]
\centering
\includegraphics[width=0.4\linewidth]{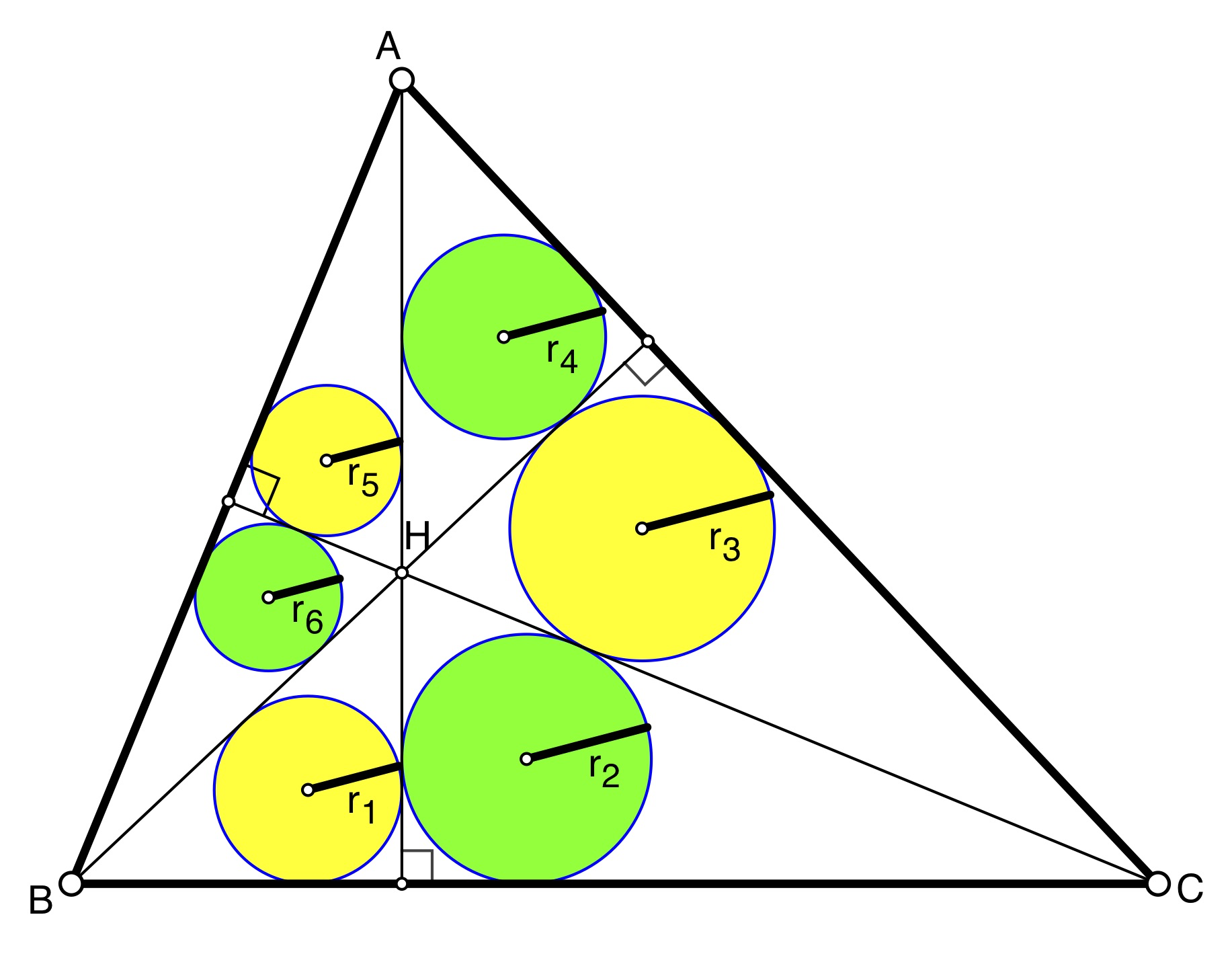}
\caption{$r_1r_3r_5=r_2r_4r_6$}
\label{fig:incirclesH}
\end{figure}

In this paper, we will find similar results using excircles instead of incircles.
When the cevians through a point $P$ interior to a triangle $ABC$ are drawn,
many smaller triangles of various sizes are formed. These triangles have three
excircles each. In this paper, we only investigate two configurations of six excircles.
These two configurations are shown in Figure \ref{fig:configurations}.
Note that in configuration 1, the circle with radius $r_1$ is an excircle of $\triangle BAD$.
In configuration 2, the circle with radius $r_1$ is an excircle of $\triangle BPD$.

\begin{figure}[h!t]
\centering
\includegraphics[width=0.9\linewidth]{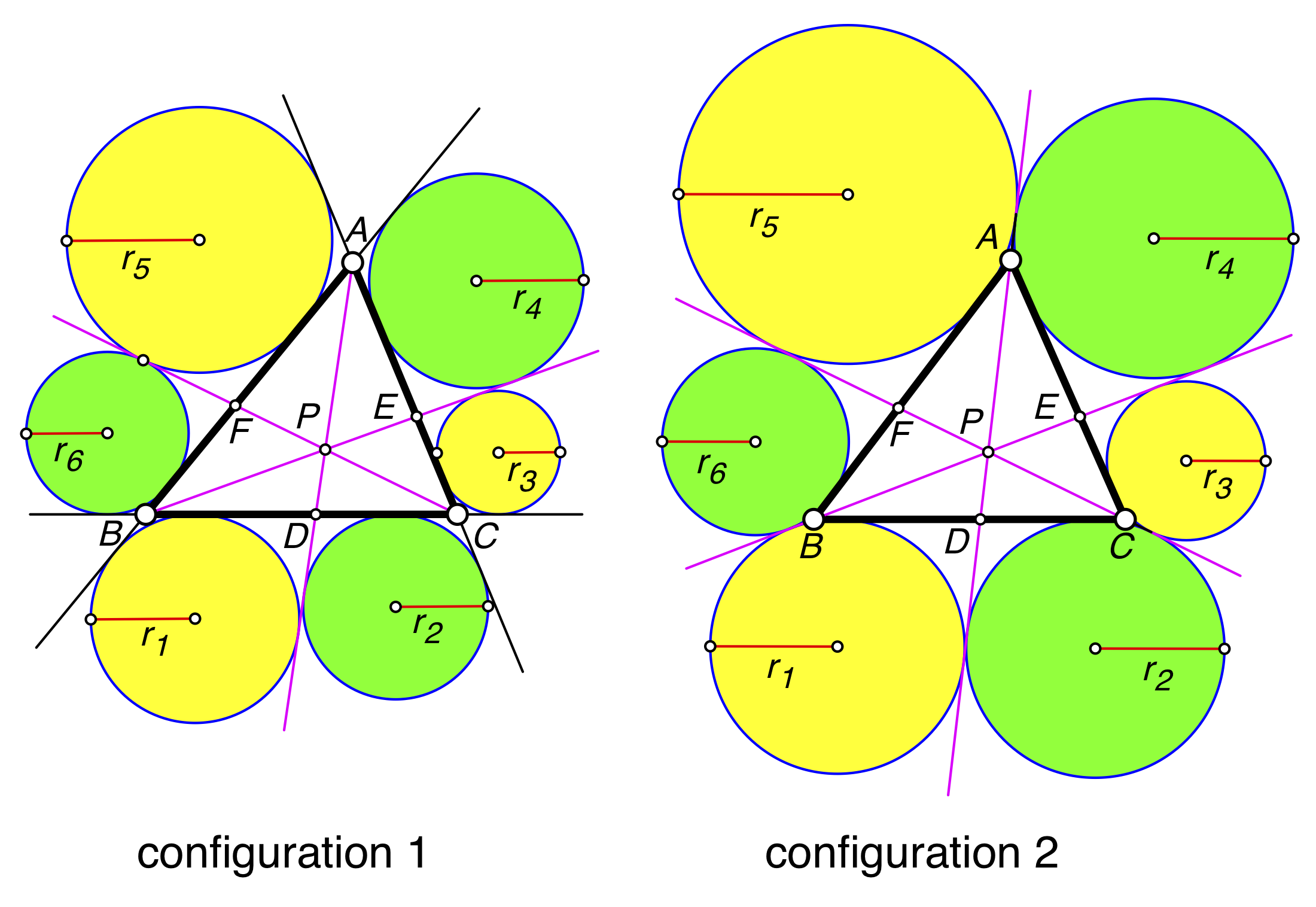}
\caption{configurations}
\label{fig:configurations}
\end{figure}

\bigskip
\textbf{Notation}: If $X$ and $Y$ are points, then we use the notation $XY$ to denote either the line segment joining $X$ and $Y$ or the length of that line segment, depending on the context.
The notation [XYZ] denotes the area of $\triangle XYZ$.

\bigskip 
\section{The Orthocenter}

When $P$ is the orthocenter of $\triangle ABC$, we have two results, depending upon which excircles are used.

\begin{theorem}
\label{thm:Orthocenter1}
Suppose the orthocenter, $H$, of $\triangle ABC$ lies inside $\triangle ABC$.
Let $r_1$ through $r_6$ be the radii of six circles tangent to the sides of $\triangle ABC$ and the cevians through $H$ situated as shown in Figure~\ref{fig:orthocenter1}.
Then $r_1r_3r_5=r_2r_4r_6$.
\end{theorem}

\begin{figure}[h!t]
\centering
\includegraphics[width=0.5\linewidth]{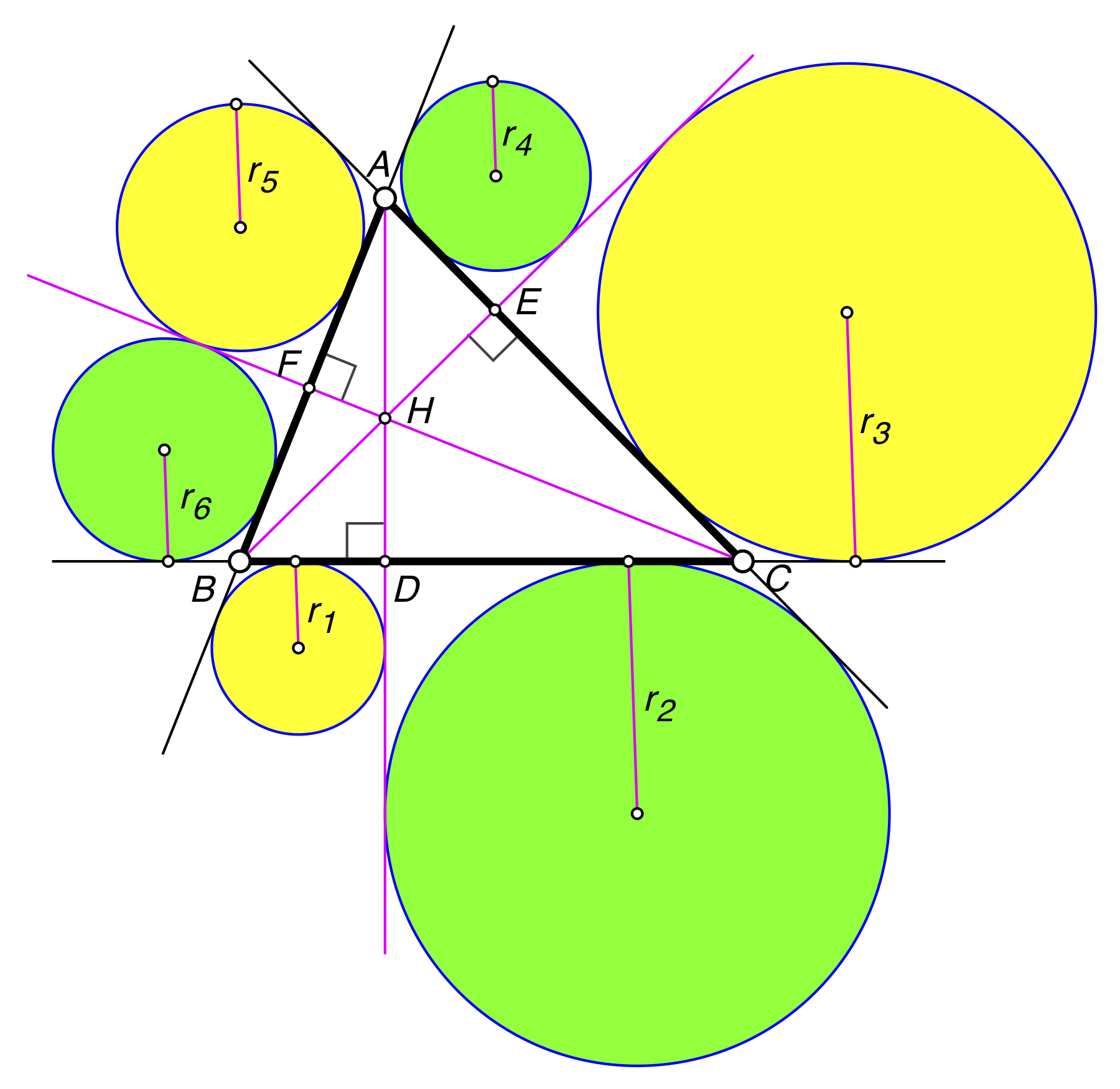}
\caption{$r_1r_3r_5=r_2r_4r_6$}
\label{fig:orthocenter1}
\end{figure}

\begin{proof}
Note that the figure consisting of $\triangle ABD$ together with the circle with radius~$r_1$ is similar
to the figure consisting of $\triangle CBF$ together with the circle with radius~$r_6$.
Corresponding lengths in similar figures are in proportion, so $r_1/AD=r_6/CF$.
Similarly, we find $r_3/BE=r_2/AD$ and $r_5/CF=r_4/BE$.
Therefore
$$\frac{r_1}{r_6}\cdot\frac{r_3}{r_2}\cdot\frac{r_5}{r_4}
=\frac{AD}{CF}\cdot\frac{BE}{AD}\cdot\frac{CF}{BE}=1$$
which implies
$r_1r_3r_5=r_2r_4r_6$.
\end{proof}

A similar result occurs when different excircles are used.

\begin{theorem}
\label{thm:Orthocenter2}
Suppose the orthocenter, $H$, of $\triangle ABC$ lies inside $\triangle ABC$.
Let $r_1$ through $r_6$ be the radii of six circles tangent to the sides of $\triangle ABC$ and the cevians through $H$ situated as shown in Figure~\ref{fig:orthocenter2}.
Then $r_1r_3r_5=r_2r_4r_6$.
\end{theorem}

\bigskip
\begin{figure}[h!t]
\centering
\includegraphics[width=0.5\linewidth]{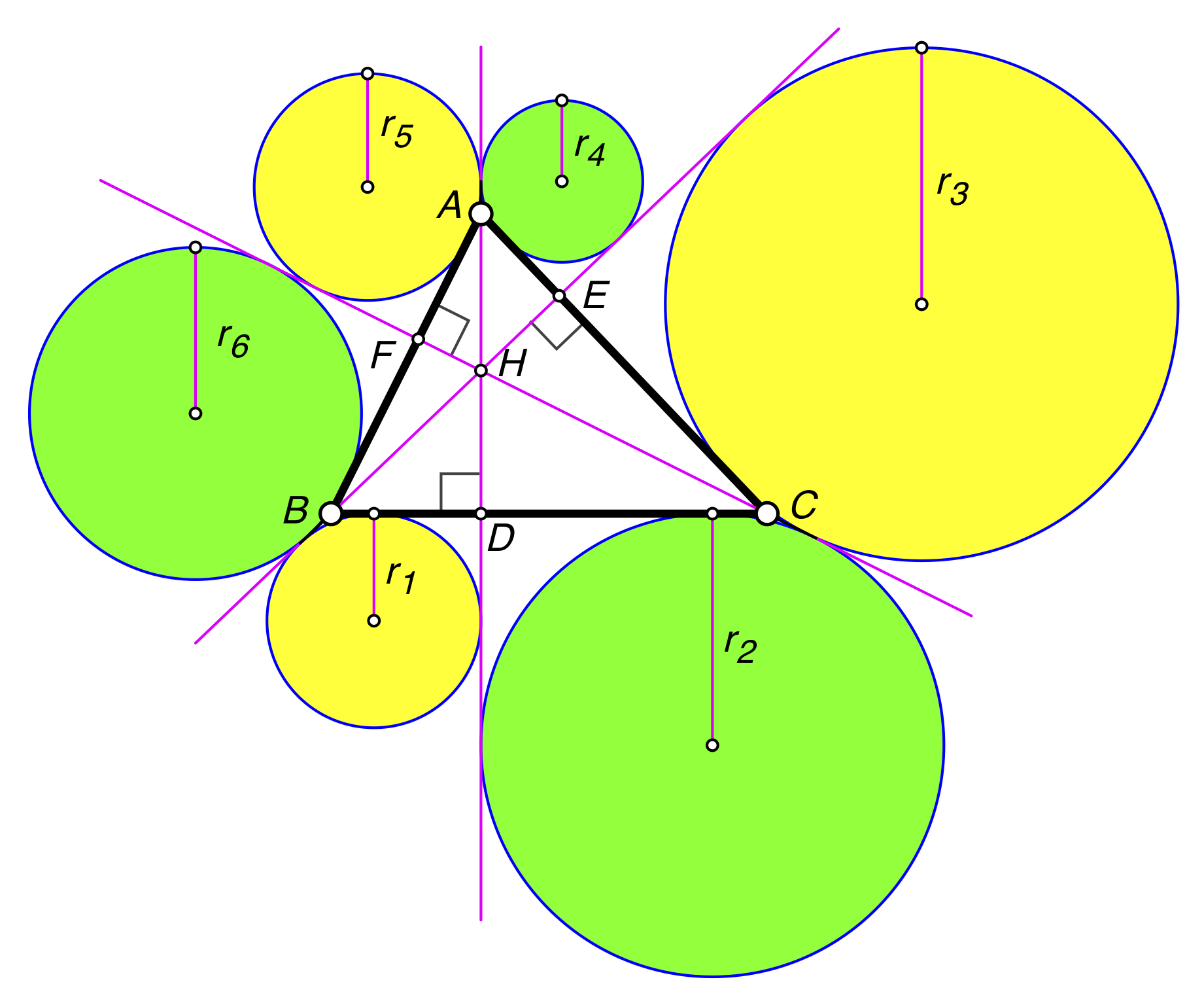}
\caption{$r_1r_3r_5=r_2r_4r_6$}
\label{fig:orthocenter2}
\end{figure}

\begin{proof}
Note that the figure consisting of $\triangle HBD$ together with the circle with radius $r_1$ is similar
to the figure consisting of $\triangle HAE$ together with the circle with radius $r_4$.
Corresponding lengths in similar figures are in proportion, so $r_1/BH=r_4/AH$.
Similarly, we find $r_3/CH=r_6/BH$ and $r_5/AH=r_2/CH$.
Therefore
$$\frac{r_1}{r_4}\cdot\frac{r_3}{r_6}\cdot\frac{r_5}{r_2}
=\frac{BH}{AH}\cdot\frac{CH}{BH}\cdot\frac{AH}{CH}=1$$
which implies
$r_1r_3r_5=r_2r_4r_6$.
\end{proof}

\bigskip 
\section{The Centroid}

Next, we will consider the case when the interior point $P$ is the centroid. We start with a lemma.

\begin{lemma}
\label{lemma:perims}
Let $M$ be the centroid of $\triangle ABC$ and let the medians be $AD$, $BE$, and $CF$.
Label the segments $BD$, $DC$, $CE$, $EA$, $AF$, and $FB$ with the numbers from 1 to 6 as shown in
Figure \ref{fig:medians}.
Let $s_i$ be the semiperimeter of the triangle formed by the segment labeled $i$ and the vertex of $\triangle ABC$ opposite that segment.
Then
\begin{equation}
\label{eq:perims}
s_1+s_3+s_5=s_2+s_4+s_6.
\end{equation}
\end{lemma}

\bigskip
\begin{figure}[h!t]
\centering
\includegraphics[width=0.4\linewidth]{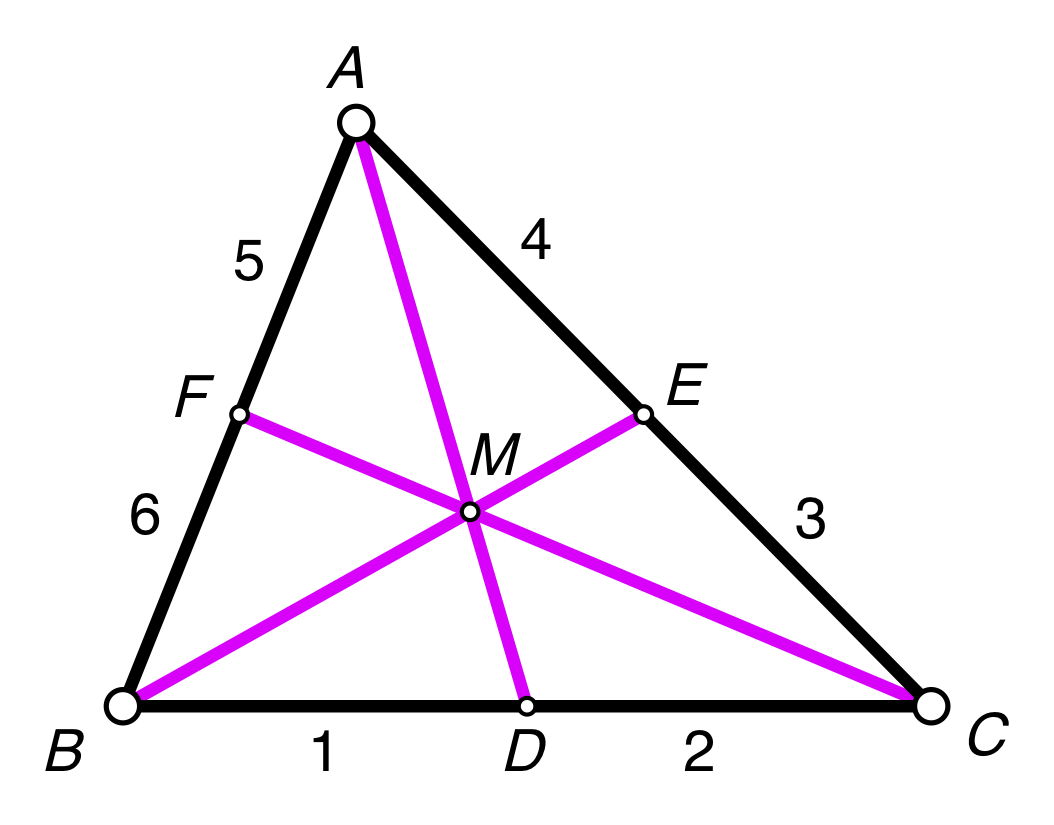}
\caption{medians}
\label{fig:medians}
\end{figure}

\begin{proof}
We have the following six equations for the perimeters of the six triangles.
$$
\begin{aligned}
2s_1&=BD+AD+AB,\qquad&2s_2&=DC+AD+CA,\\
2s_3&=CE+BE+BC,&2s_4&=EA+BE+AB,\\
2s_5&=AF+CF+CA,&2s_6&=FB+CF+BC.
\end{aligned}
$$
Thus,
$2s_1+2s_3+2s_5-(2s_2+2s_4+2s_6)=(BD-DC)+(CE-EA)+(AF-FB)=0$
and the lemma follows.
\end{proof}

Recall the well-known formula for the length of the radius of an excircle.
If the sides of a triangle have lengths $a$, $b$, and $c$, and the semiperimeter is $s$,
then the radius of the excircle that touches the side of length $a$ is $K/(s-a)$,
where $K$ is the area of the triangle. See, for example, \cite[p.~79]{Altshiller-Court}.

We can now state our results.

\begin{theorem}
\label{thm:Centroid1}
Let $M$ be the centroid of $\triangle ABC$ and let the medians be $AD$, $BE$, and $CF$.
Let $r_1$ through $r_6$ be the radii of six circles tangent to the sides of $\triangle ABC$ and the cevians through $M$ situated as shown in Figure~\ref{fig:centroid1}.
Then
$$\frac{1}{r_1}+\frac{1}{r_3}+\frac{1}{r_5}=\frac{1}{r_2}+\frac{1}{r_4}+\frac{1}{r_6}.$$
\end{theorem}

\bigskip
\begin{figure}[h!t]
\centering
\includegraphics[width=0.5\linewidth]{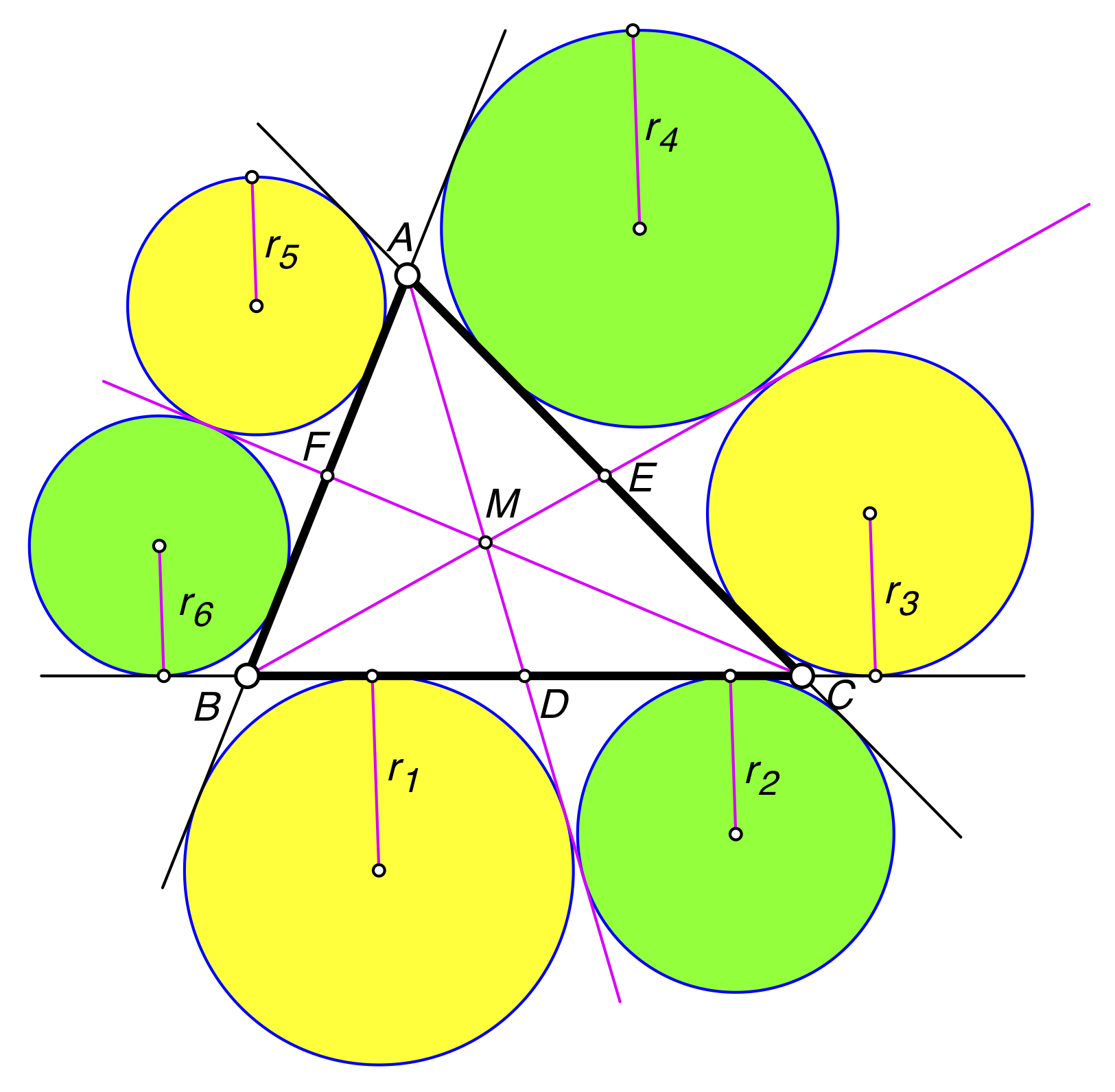}
\caption{}
\label{fig:centroid1}
\end{figure}

\begin{proof}
The circle with radius $r_i$ is an excircle of a triangle as shown in Figure \ref{fig:centroid1}.
Let $s_i$ be the semiperimeter of that triangle
and let $K_i$ be the area of that triangle.
Note that each $K_i$ is half the area of $\triangle ABC$. Denote this common value by $K$.
From the formula for the length of the radius of an excircle, we have
$$\frac{1}{r_1}=\frac{s_1-BD}{K},\qquad\frac{1}{r_3}=\frac{s_3-CE}{K},\qquad
\frac{1}{r_5}=\frac{s_5-AF}{K}.$$
Adding gives
$$\frac{1}{r_1}+\frac{1}{r_3}+\frac{1}{r_5}=\frac{s_1+s_3+s_5-s}{K}$$
where $s=BD+CE+AF$ is the semiperimeter of $\triangle ABC$.
In the same manner, we find
$$\frac{1}{r_2}+\frac{1}{r_4}+\frac{1}{r_6}=\frac{s_2+s_4+s_6-s}{K}.$$
From Lemma \ref{lemma:perims}, $s_1+s_3+s_5=s_2+s_4+s_6$.
Thus,
$$\frac{1}{r_1}+\frac{1}{r_3}+\frac{1}{r_5}=\frac{1}{r_2}+\frac{1}{r_4}+\frac{1}{r_6}$$
as required.
\end{proof}

A similar result occurs when different excircles are used.

\begin{theorem}
\label{thm:Centroid2}
Let $M$ be the centroid of $\triangle ABC$ and let the medians be $AD$, $BE$, and $CF$.
Let $r_1$ through $r_6$ be the radii of six circles tangent to the sides of $\triangle ABC$ and the cevians through $M$ situated as shown in Figure~\ref{fig:centroid2}.
Then
$$\frac{1}{r_1}+\frac{1}{r_3}+\frac{1}{r_5}=\frac{1}{r_2}+\frac{1}{r_4}+\frac{1}{r_6}.$$
\end{theorem}

\bigskip
\begin{figure}[h!t]
\centering
\includegraphics[width=0.45\linewidth]{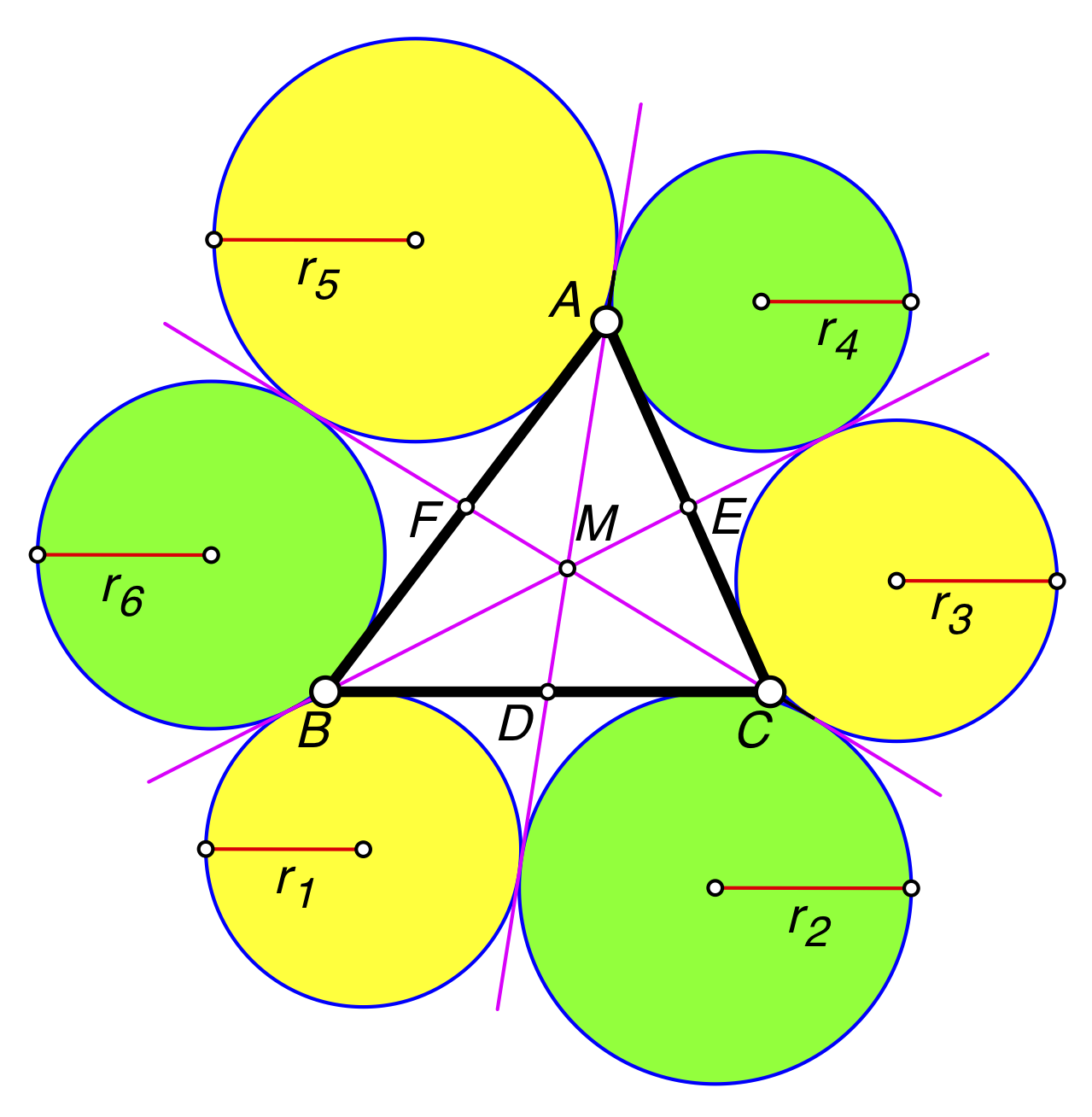}
\caption{}
\label{fig:centroid2}
\end{figure}

\begin{proof}
The proof is similar to the proof of Theorem \ref{thm:Centroid1}.
In this case, the circles are excircles of triangles $BMD$, $CMD$, $CME$, $AME$,
$AMF$, and $BMF$.
These six triangles have the same area and their semiperimeters
satisfy equation (\ref{eq:perims}). The details are omitted.
\end{proof}

\bigskip
\section{The Gergonne Point}

Suppose the incircle of $\triangle ABC$ touches the sides $BC$, $CA$, and $AB$ at points $D$, $E$, and $F$, respectively (Figure \ref{fig:GergonnePoint}). Then the cevians $AD$, $BE$, and $CF$, meet at a point, $G_e$,
known as the Gergonne Point of the triangle \cite[p.~160]{Altshiller-Court}.

\begin{figure}[h!t]
\centering
\includegraphics[width=0.4\linewidth]{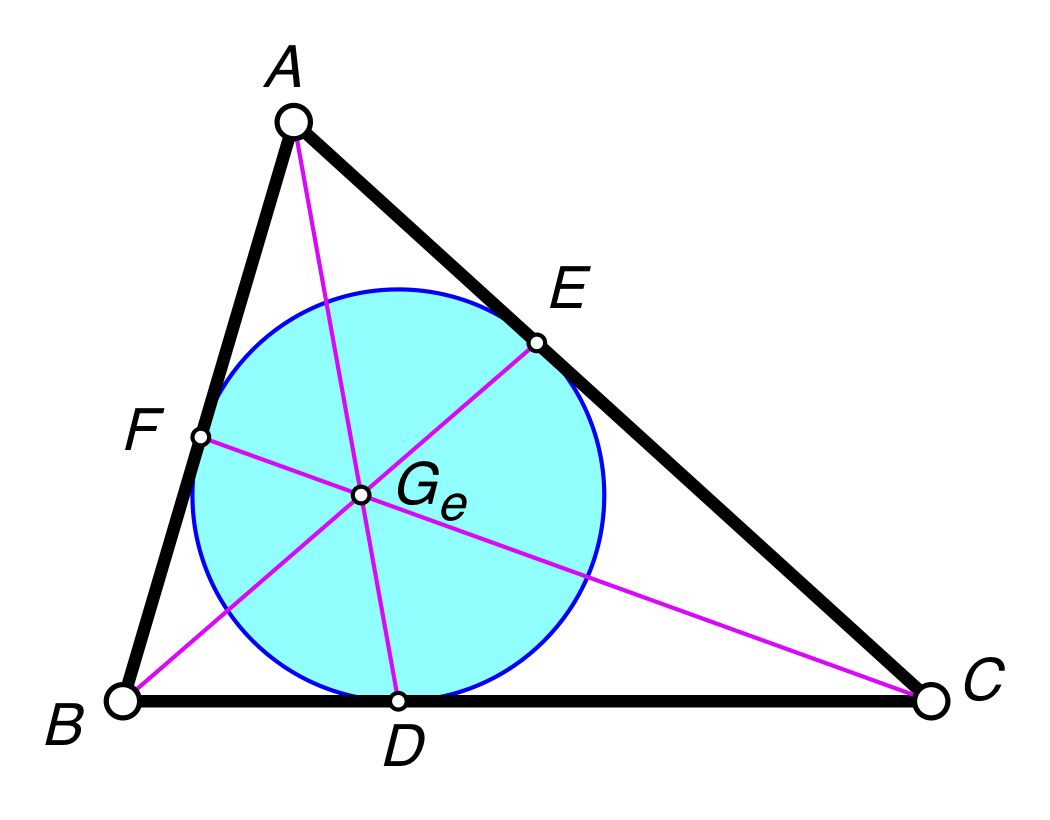}
\caption{Gergonne Point}
\label{fig:GergonnePoint}
\end{figure}

We will now find relationships between certain excircles associated with a
triangle and the three cevians through its Gergonne Point.
We start with a lemma.

\begin{lemma}
\label{lemma:Gergonne}
Let $D$ be the contact point of the incircle of $\triangle ABC$ with side $BC$
(Figure \ref{fig:GergonneLemma}).
The excircle of $\triangle ABD$ that touches side $BD$ has radius $r_1$.
The excircle of $\triangle ADC$ that touches side $DC$ has radius $r_2$.
Then
$$\frac{r_1}{r_2}=\frac{BD}{DC}.$$
\end{lemma}

\begin{figure}[h!t]
\centering
\includegraphics[width=0.5\linewidth]{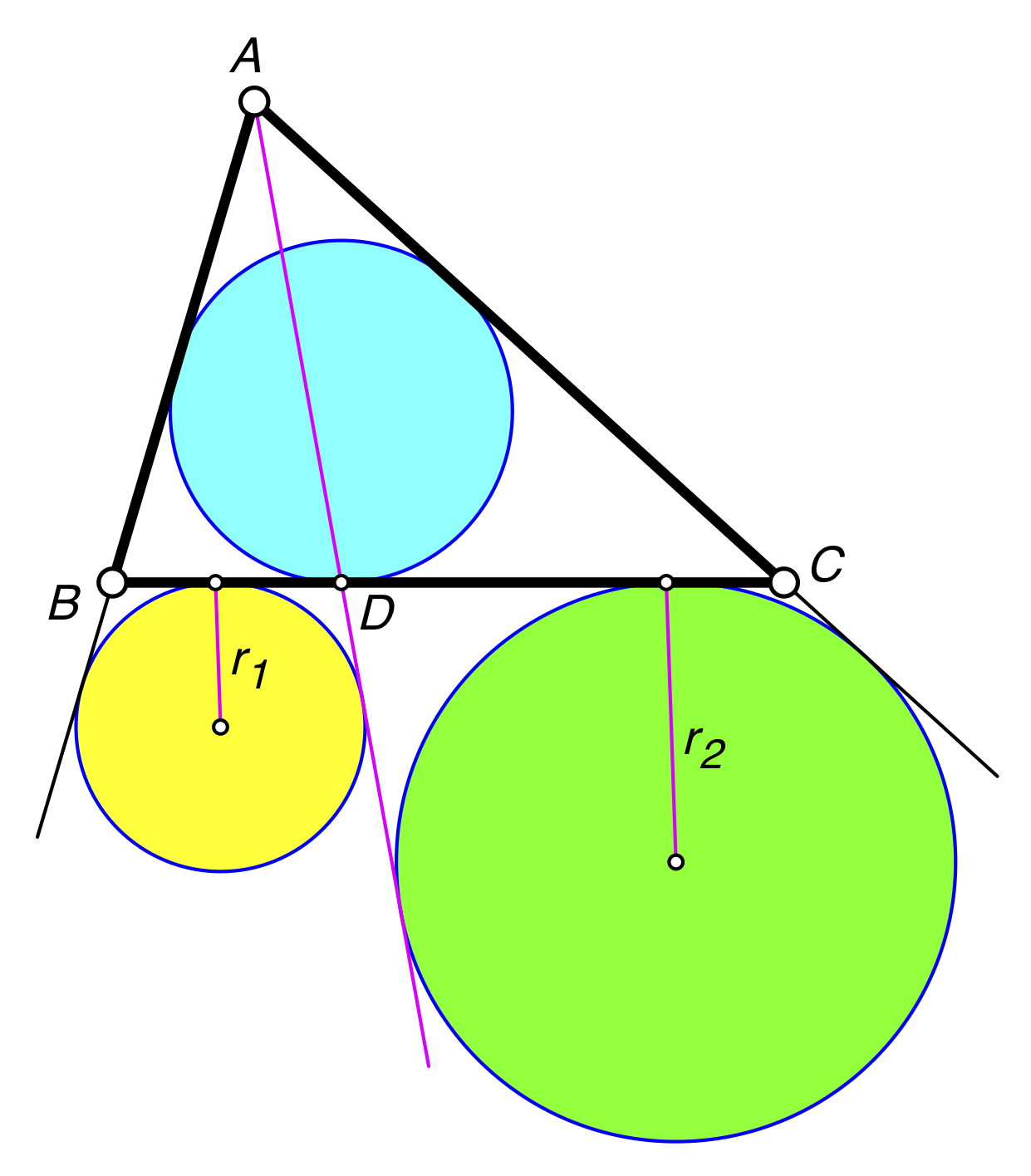}
\caption{$\displaystyle\frac{r_1}{r_2}=\frac{BD}{DC}$}
\label{fig:GergonneLemma}
\end{figure}

The following proof is due to Duca \cite{Duca}.

\begin{proof}
Using the formula for the radius of an excircle, we have
\begin{equation}
\label{eq:GergonneLemma}
r_1=\frac{2[ABD]}{AB+AD-BD}\qquad\textrm{and}\qquad r_2=\frac{2[ADC]}{CA+AD-DC}.
\end{equation}
By a well-known property of the incircle of a triangle \cite[p.~87]{Altshiller-Court},
$BD=s-CA$ and $DC=s-AB$,
where $s=(AB+BC+CA)/2$.
Thus $CA+BD=AB+DC$, which implies $AB+AD-BD=CA+AD-DC$.
Therefore, the two denominators in equation (\ref{eq:GergonneLemma}) are equal.
Hence $r_1/r_2=[ABD]/[ADC]$. Since triangles $ABD$ and $ADC$ have the same altitude from $A$,
the ratio of their areas will be proportional to the ratio of their bases.
Consequently,
$$\frac{r_1}{r_2}=\frac{[ABD]}{[ADC]}=\frac{BD}{DC}.$$
\end{proof}

We can now easily prove the following theorem.

\begin{theorem}
Let $r_1$ through $r_6$ be the radii of six circles tangent to the sides of $\triangle ABC$ and the cevians through the Gergonne point situated as shown in Figure~\ref{fig:Gergonne}.
Then
$r_1r_3r_5=r_2r_4r_6$.
\end{theorem}

\bigskip
\begin{figure}[h!t]
\centering
\includegraphics[width=0.4\linewidth]{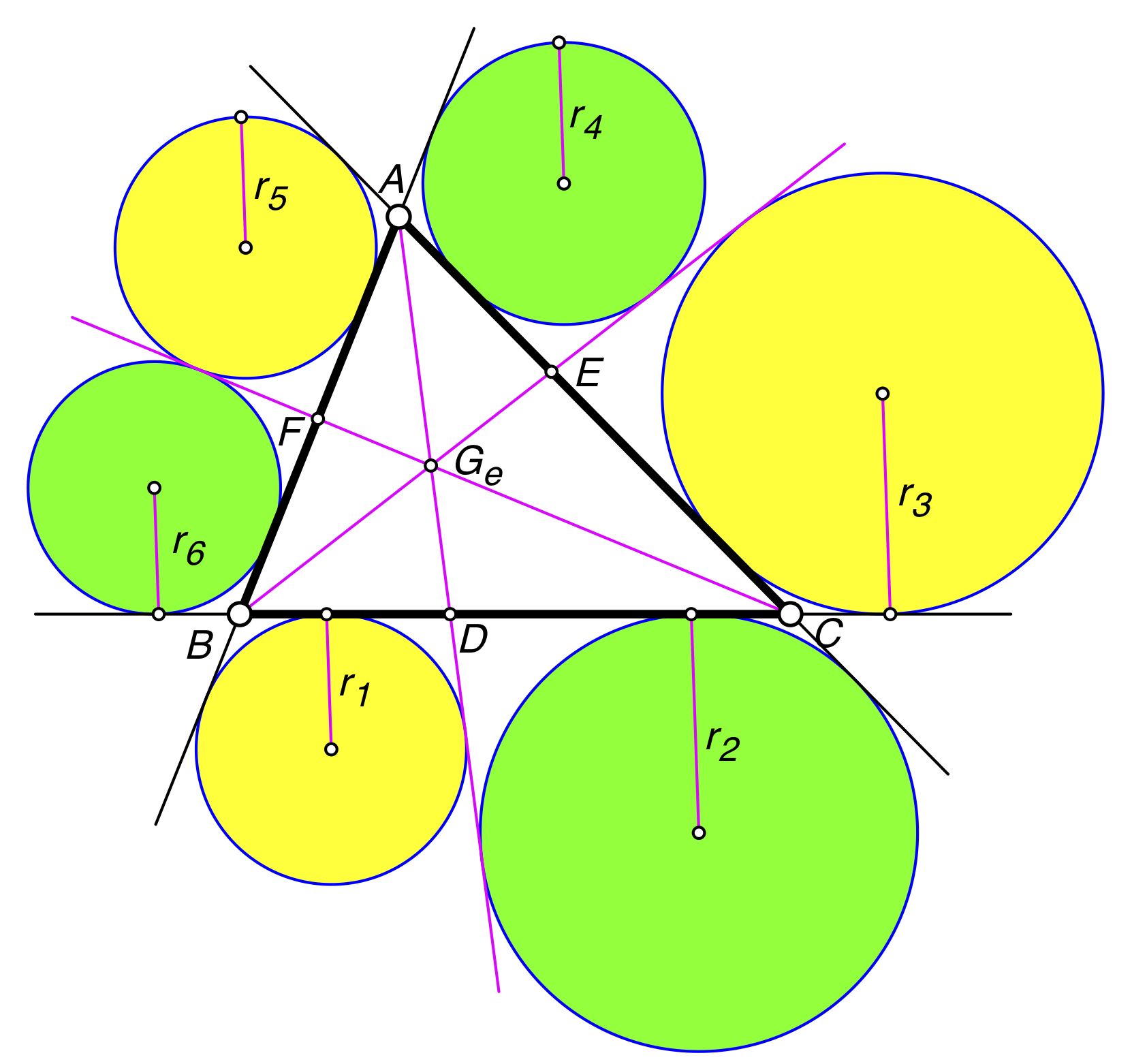}
\caption{$r_1r_3r_5=r_2r_4r_6$}
\label{fig:Gergonne}
\end{figure}

\begin{proof}
By Lemma \ref{lemma:Gergonne},
$$\frac{r_1}{r_2}\cdot\frac{r_3}{r_4}\cdot\frac{r_5}{r_6}=
\frac{BD}{DC}\cdot\frac{CE}{EA}\cdot\frac{AF}{FB}.$$
The expression on the right is equal to 1 by Ceva's Theorem. Thus, we have
$r_1r_3r_5=r_2r_4r_6$.
\end{proof}

\section{The Nagel Point}

\newcommand{\Na}{N_a}

Suppose the excircles of $\triangle ABC$ touch the sides $BC$, $CA$, and $AB$ at points $D$, $E$, and $F$, respectively, as shown in Figure \ref{fig:NagelPoint}. Then the cevians $AD$, $BE$, and $CF$ meet at a point
known as the Nagel Point of the triangle \cite[p.~160]{Altshiller-Court}.

\begin{figure}[h!t]
\centering
\includegraphics[width=0.4\linewidth]{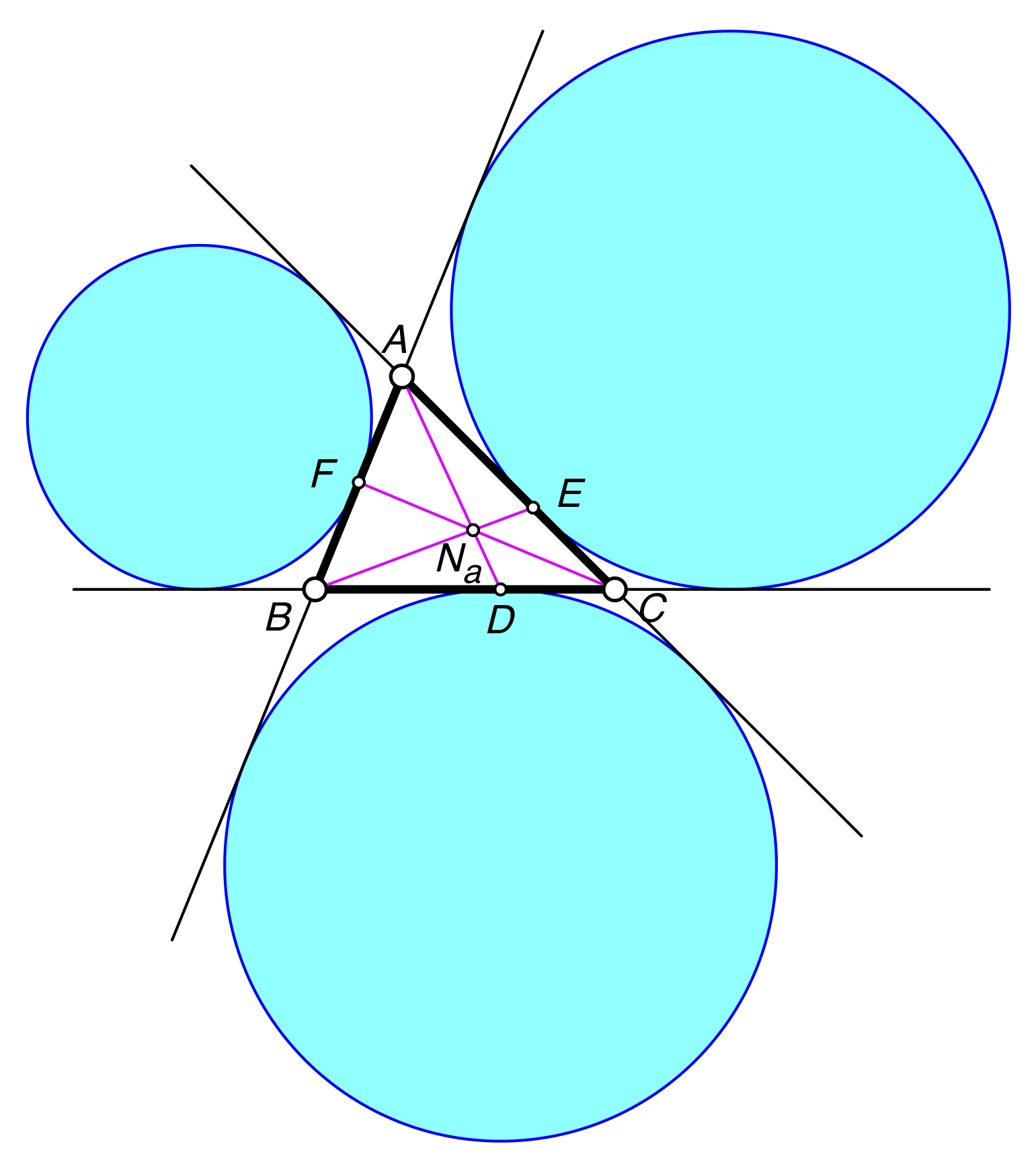}
\caption{Nagel Point}
\label{fig:NagelPoint}
\end{figure}

\begin{lemma}
\label{lemma:1}
Let $\Na$ be the Nagel point of $\triangle ABC$. The cevians $AD$, $BE$, and $CF$
through $\Na$ divide $\triangle ABC$ into six small triangles numbered from 1 to 6 as shown
in Figure \ref{fig:Nagel-triangles}. Let $K_i$ be the area of triangle $i$.
Then $K_1K_3K_5=K_2K_4K_6$.
\end{lemma}

\begin{figure}[h!t]
\centering
\includegraphics[width=0.38\linewidth]{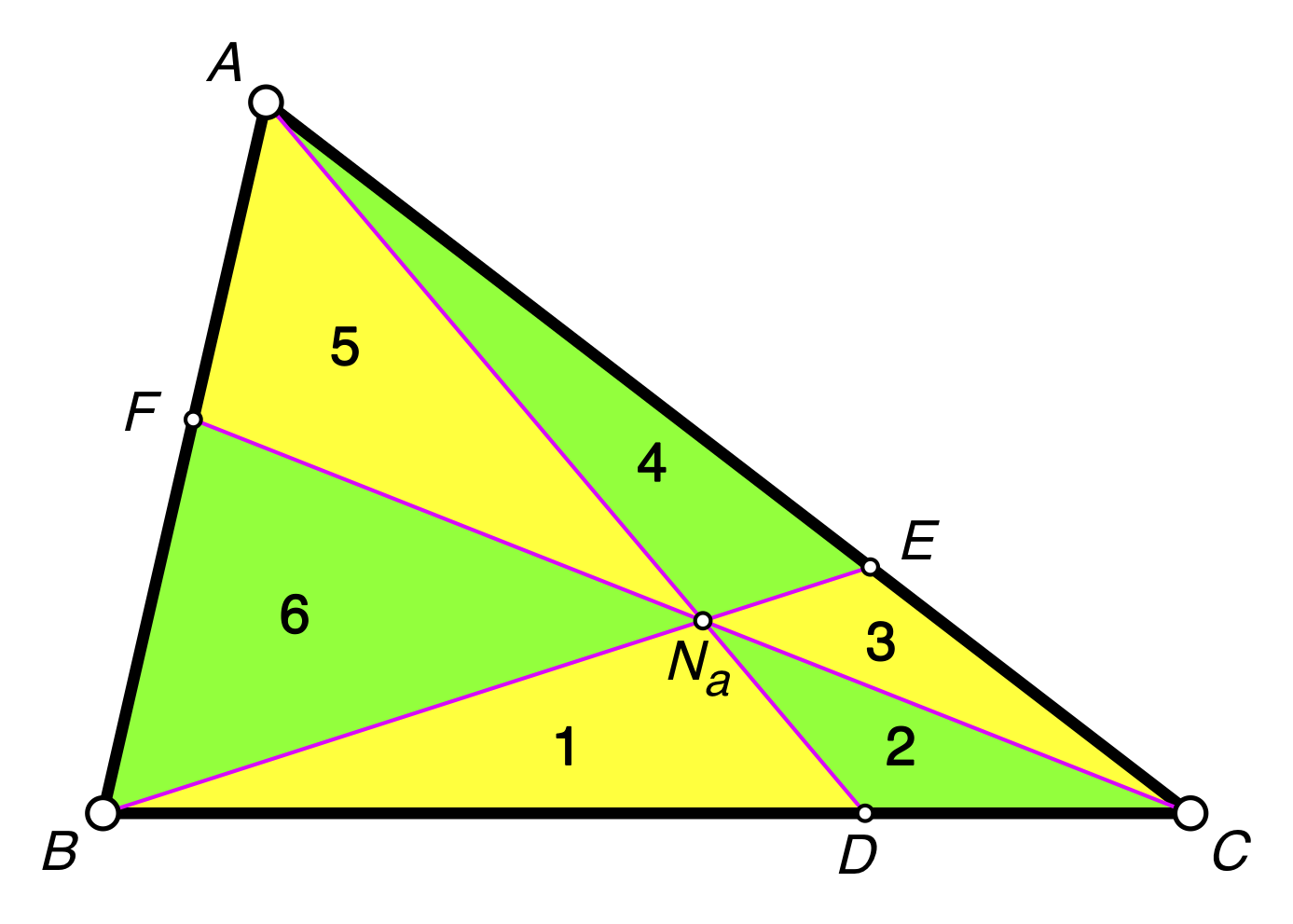}
\caption{Nagel Point}
\label{fig:Nagel-triangles}
\end{figure}

\begin{proof}
This is a special case of Theorem 7.4 from \cite{Rabinowitz}.
\end{proof}

\begin{lemma}
\label{lemma:2}
Let $\Na$ be the Nagel point of $\triangle ABC$. The cevians $AD$, $BE$, and $CF$
through $\Na$ divide $\triangle ABC$ into six small triangles numbered from 1 to 6 as shown
in Figure \ref{fig:Nagel-triangles}. Let $s_i$ be the semiperimeter of triangle $i$.
Then $s_1s_3s_5=s_2s_4s_6$.
\end{lemma}

\begin{proof}
Let $a=BC$, $b=CA$, $c=AB$, and $s=(a+b+c)/2$.
It is well-known that $AF=DC=s-b$, $FB=CE=s-a$, and $BD=EA=s-c$ \cite[p.~88]{Altshiller-Court}.
By Stewart's Theorem, we can compute the lengths of cevians $AD$, $BE$, and $CF$ in $\triangle ABC$. We get
$$
AD=\sqrt{s^2-\frac{4K^2}{a(s-a)}},\quad
BE=\sqrt{s^2-\frac{4K^2}{b(s-b)}},\quad
CF=\sqrt{s^2-\frac{4K^2}{c(s-c)}},
$$
where $K$ is the area of $\triangle ABC$ and $s$ is its semiperimeter.
We can then use the Theorem of Menelaus on $\triangle ADC$ with traversal $BE$
to find the ratio of $A\Na$ to $\Na D$. We can use the same procedure to find how $\Na$
divides the other cevians. We get
$$
\frac{A\Na}{\Na D}=\frac{a}{s-a},\quad
\frac{B\Na}{\Na E}=\frac{b}{s-b},\quad
\frac{C\Na}{\Na F}=\frac{c}{s-c}.\quad
$$

These formulas can also be found in \cite{Krishna}.
This allows us to compute the lengths of $A\Na$, $\Na D$, $B\Na$, $\Na E$, $C\Na$, and $\Na F$
in terms of $a$, $b$, $c$, $s$, and $K$. This, in turn, gives us expressions
for $s_1$, $s_2$, $s_3$, $s_4$, $s_5$, and $s_6$.
Simplifying $s_1s_3s_5-s_2s_4s_6$ using a computer algebra system shows that the result is 0.
\end{proof}

\begin{theorem}
\label{thm:Nagel1}
Let $\Na$ be the Nagel point of $\triangle ABC$. The cevians $AD$, $BE$, and $CF$
through $\Na$ divide $\triangle ABC$ into six small triangles numbered from 1 to 6 as shown
in Figure \ref{fig:Nagel-triangles}. Let $r_i$ be the radius of the incircle of triangle $i$.
Then $r_1r_3r_5=r_2r_4r_6$.
\end{theorem}

\begin{proof}
Using the notation from Lemmas \ref{lemma:1} and \ref{lemma:2}, we have $r_i=K_i/s_i$, so
$$r_1r_3r_5=\frac{K_1}{s_1}\frac{K_3}{s_3}\frac{K_5}{s_5}=\frac{K_2}{s_2}\frac{K_4}{s_4}\frac{K_6}{s_6}=r_2r_4r_6$$
as required.
\end{proof}

A similar result holds for excircles.

\begin{lemma}
\label{lemma:3}
Let $\Na$ be the Nagel point of $\triangle ABC$. The cevians $AD$, $BE$, and $CF$
through $\Na$ divide $\triangle ABC$ into six small triangles numbered from 1 to 6 as shown
in Figure \ref{fig:Nagel-triangles}. Let $s_i$ be the semiperimeter of triangle $i$.
Then $$(s_1-BD)(s_3-CE)(s_5-AF)=(s_2-DC)(s_4-EA)(s_6-FB).$$
\end{lemma}

\begin{proof}
The proof is essentially the same as the proof of Lemma \ref{lemma:2}.
Expressions for all the needed lengths have already been found in terms of
$a$, $b$, $c$, $s$, and $K$.
Simplifying $(s_1-BD)(s_3-CE)(s_5-AF)-(s_2-DC)(s_4-EA)(s_6-FB)$ using a computer algebra system shows that the result is 0.
\end{proof}

\begin{theorem}
\label{thm:Nagel2}
Let $r_1$ through $r_6$ be the radii of six circles tangent to the sides of $\triangle ABC$ and the cevians through the Nagel point situated as shown in Figure~\ref{fig:Nagel}.
Then
$r_1r_3r_5=r_2r_4r_6$.
\end{theorem}

\begin{figure}[h!t]
\centering
\includegraphics[width=0.4\linewidth]{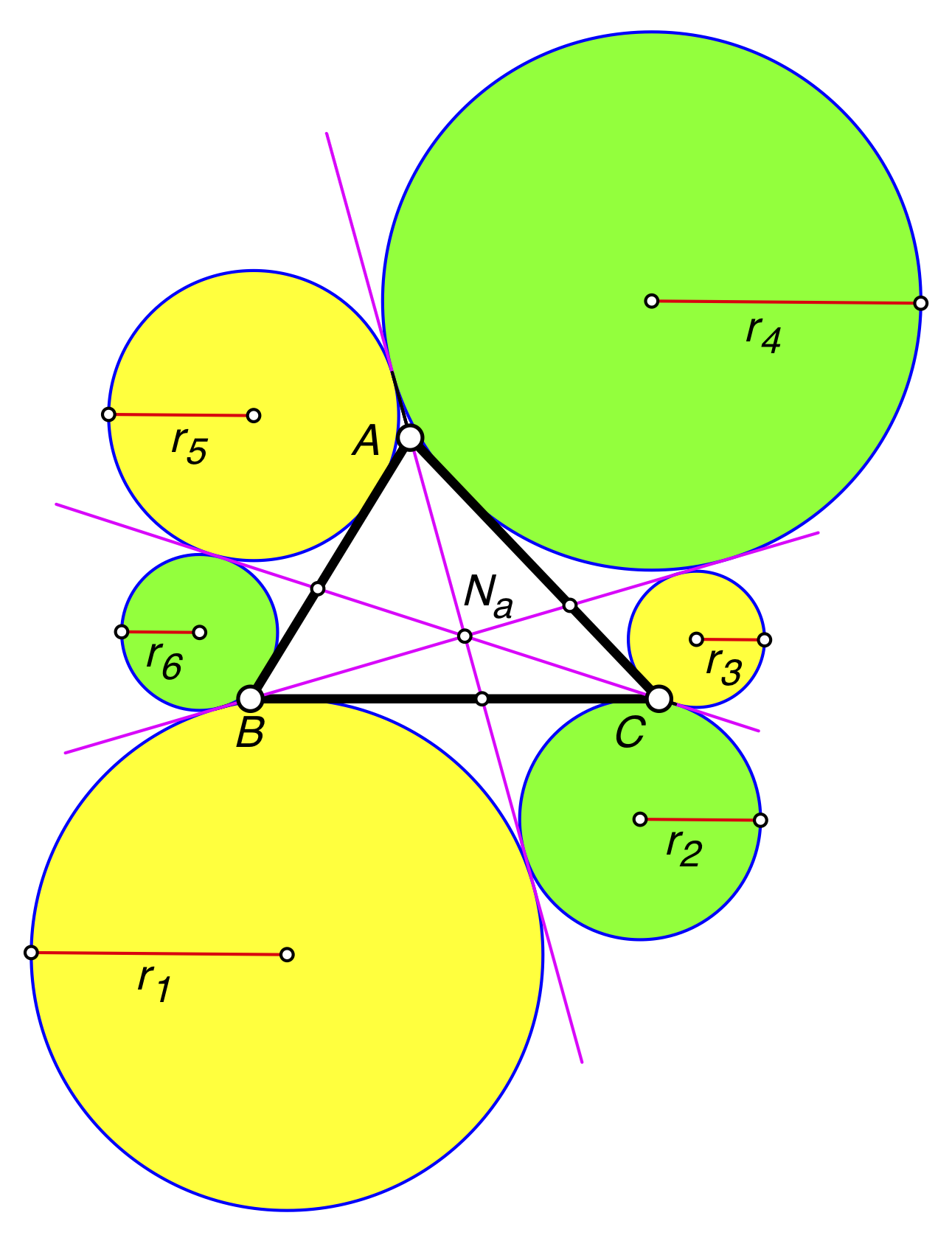}
\caption{$r_1r_3r_5=r_2r_4r_6$}
\label{fig:Nagel}
\end{figure}

\begin{proof}
The proof is the same as the proof of Theorem \ref{thm:Nagel1} only now
$\displaystyle r_i=\frac{K_i}{(s-a_i)}$
where $a_i$ is the length of the side of triangle $i$ lying along a side of $\triangle ABC$.
\end{proof}

The results of Theorems \ref{thm:Nagel1} and \ref{thm:Nagel2} are so elegant
that it is unlikely that they are true only because the complicated expressions
found in the proofs of Lemmas \ref{lemma:2} and \ref{lemma:3} just happen to
simplify to 0.

\begin{open}
Are there simple proofs of Theorems \ref{thm:Nagel1} and \ref{thm:Nagel2}
that do not involve a large amount of algebraic computation requiring
computer simplification?
\end{open}

\bigskip
\section{Relationship Between Inradii and Exradii}

We can make use of relationships between the radii of incircles
associated with a figure to find relationships between the radii of
associated excircles.

The following basic result was known to Japanese geometers of the Edo period
as evidenced by the fact that it is equivalent to a problem found in the 1823 text, \textit{Sangaku Shousen}
\cite{Ushijima}, later printed as problem 4.3.3 in \cite[p.~21]{Fukagawa}.

\begin{theorem}[Inradius/Exradius Invariant]
Let $A$ be a fixed point and let $L$ be a fixed line that does not pass through $A$.
Let $B$ and $C$ be variable points on $L$, with $B\neq C$.
Let $r$ be the inradius of $\triangle ABC$ and let $r_a$ be the radius of the excircle
that touches side $BC$ (Figure \ref{fig:inradius-exradius}). Then $\displaystyle\frac{1}{r}-\frac{1}{r_a}$
remains invariant as $B$ and $C$ vary along $L$.
\end{theorem}

\begin{figure}[h!t]
\centering
\includegraphics[width=0.5\linewidth]{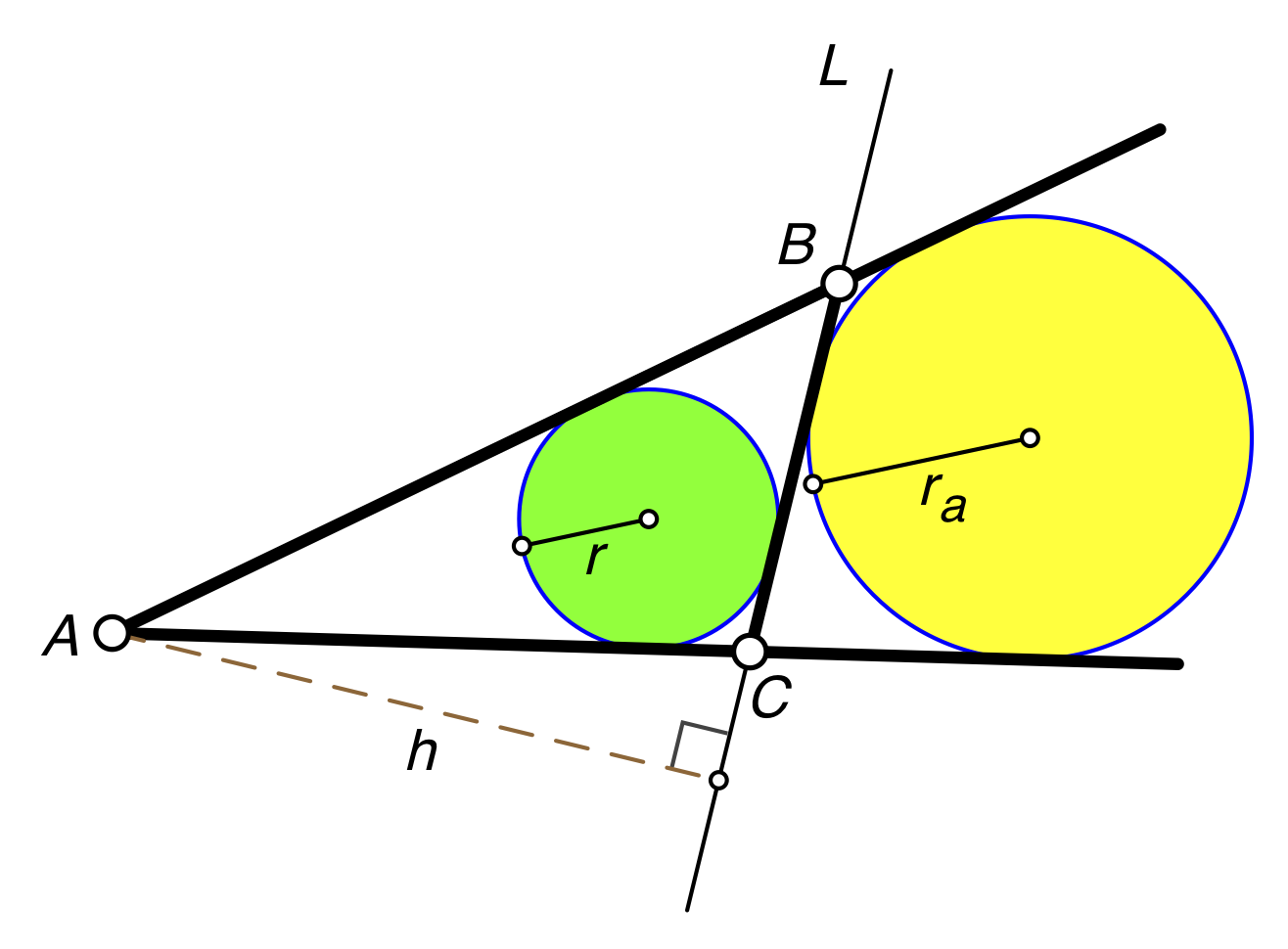}
\caption{inradius and exradius}
\label{fig:inradius-exradius}
\end{figure}

\begin{proof}
From the formulas for the length of an inradius and an exradius, we have
$$r=\frac{K}{s}\quad\mathrm{and}\quad r_a=\frac{K}{s-a},$$
where $a$ is the length of $BC$, $K$ is the area of $\triangle ABC$, and $s$ is its semiperimeter.
Thus $$\frac{1}{r}-\frac{1}{r_a}=\frac{s}{K}-\frac{s-a}{K}=\frac{a}{K}=\frac{2}{h},$$
where $h$ is the distance from $A$ to $L$. This proves the theorem since $h$ remains fixed as $B$ and $C$ vary along $L$.
\end{proof}

The following result follows immediately.

\begin{theorem}[Relationship Between Two Incircles and Two Excircles]
\label{thm:in-ex}
Let $AD$ be a cevian of $\triangle ABC$.
Four circles are tangent to the sides of the triangle and the cevian as
shown in Figure \ref{fig:in-ex-corollary}.
Then
$$\frac{1}{r_1}+\frac{1}{r_4}=\frac{1}{r_2}+\frac{1}{r_3}.$$
\end{theorem}

\begin{figure}[h!t]
\centering
\includegraphics[width=0.5\linewidth]{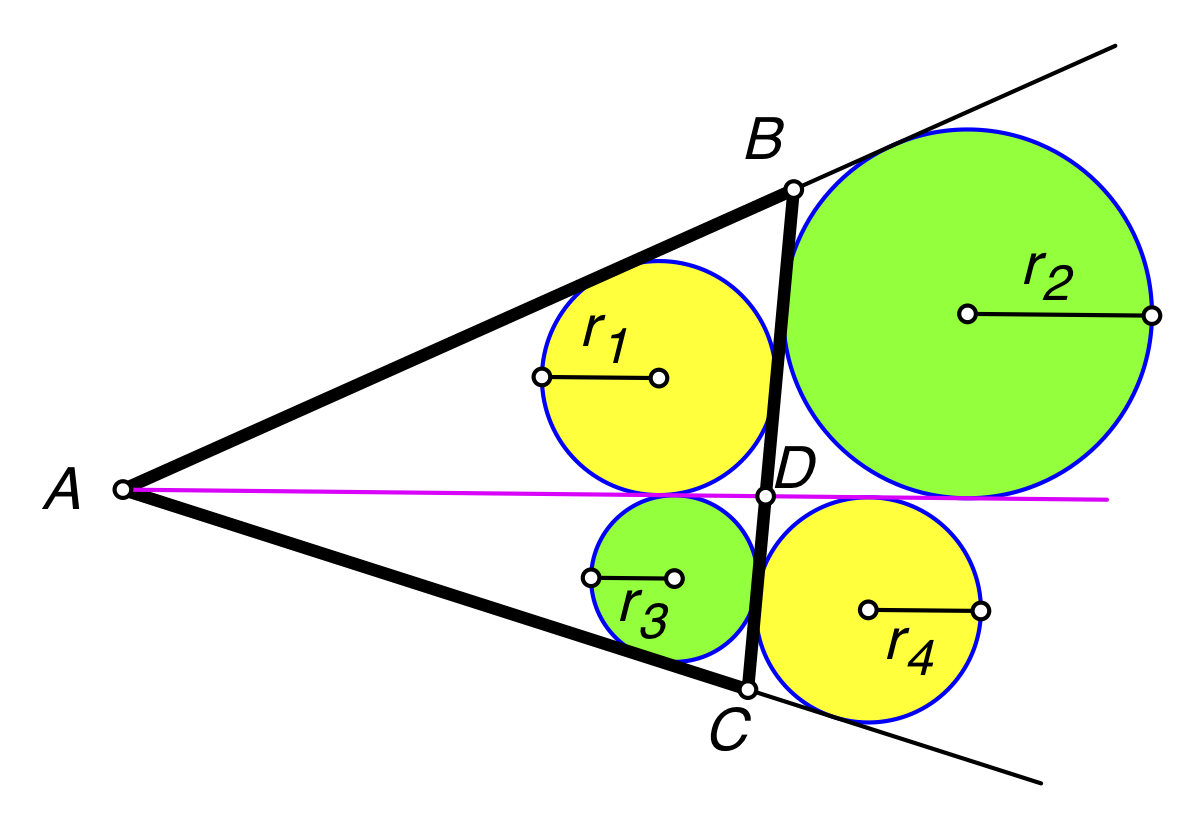}
\caption{relationship between two incircles and two excircles}
\label{fig:in-ex-corollary}
\end{figure}

There is another nice relationship between inradii and exradii.

We start with a lemma.

\begin{lemma}[Inradius/Exradius Fixed Angle Invariant]
\label{lemma:angle-invariant}
Let $A$ be a fixed point and let $L_1$ and $L_2$ be distinct fixed rays starting at $A$.
Let $B$ and $C$ be variable points on $L_1$ and $L_2$, respectively, neither
coinciding with $A$.
Let $r$ be the inradius of $\triangle ABC$ and let $r_a$ be the radius of the excircle
that touches side $BC$ (Figure~\ref{fig:in-ex-angle}).
Let $K$ be the area of $\triangle ABC$.
Then $rr_a/K$ remains invariant as $B$ and $C$ vary.
\end{lemma}

\begin{figure}[h!t]
\centering
\includegraphics[width=0.5\linewidth]{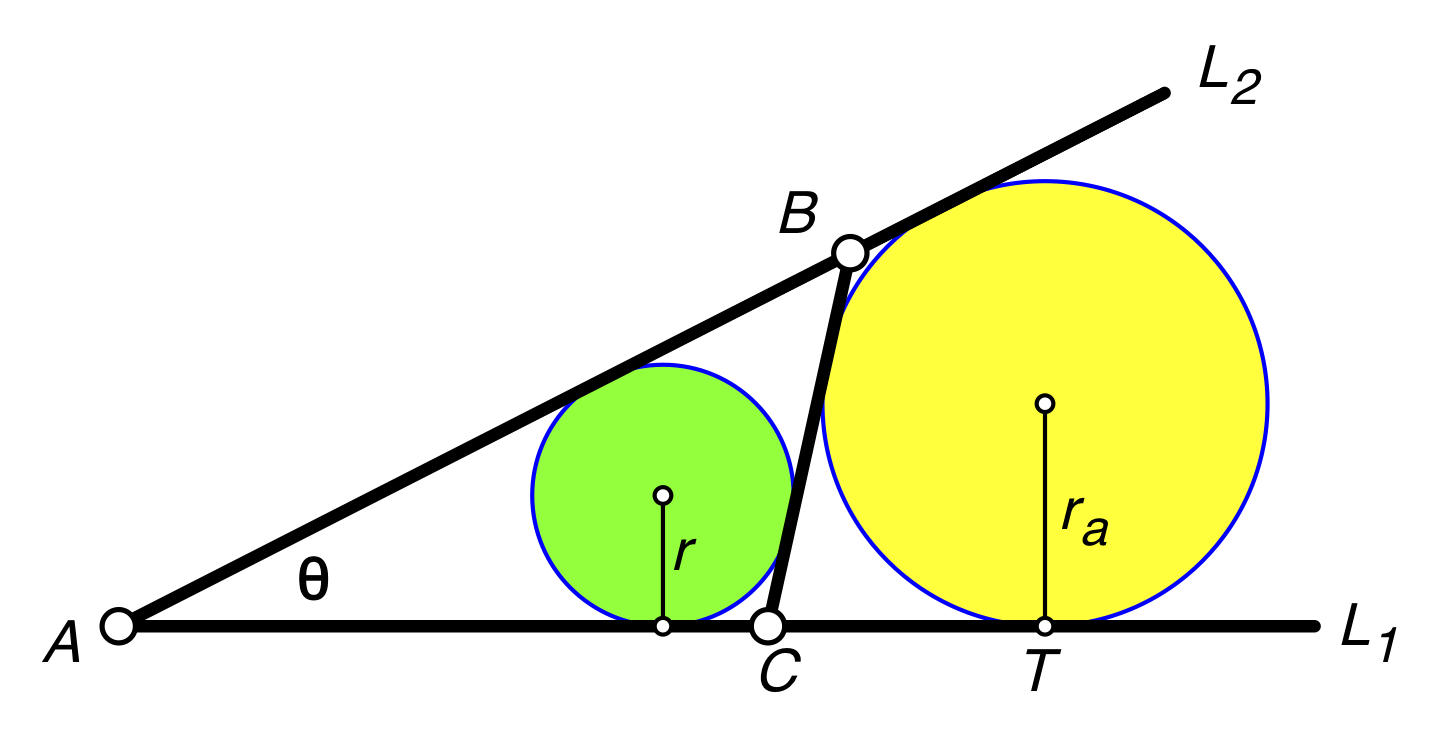}
\caption{inradius and exradius}
\label{fig:in-ex-angle}
\end{figure}

\begin{proof}
Let $T$ be the contact point of the excircle with $L_1$.
It is known that $AT=s$, where $s$ is the semiperimeter of $\triangle ABC$
\cite[Theorem~158]{Altshiller-Court}.
Let $\angle BAC=\theta$. Note that the bisector of $\angle BAC$
passes through the centers of the two circles.
Since $K=rs$, we have
$$\frac{rr_a}{K}=\frac{rr_a}{rs}=\frac{r_a}{s}=\frac{r_a}{AT}=\tan\frac{\theta}{2}.$$
Thus, $rr_a/K$ is invariant because the angle $\theta$ is fixed.
\end{proof}

\begin{theorem}[Relationship Between Six Incircles and Six Excircles]
\label{thm:in-ex6}
Let $P$ be a point inside $\triangle ABC$.
The cevians through $P$ divide $\triangle ABC$ into six small triangles,
named $T_1$ through $T_6$ as shown in Figure~\ref{fig:sixTriangles}.
Let $r_i$ be the inradius of $T_i$.
Let $R_i$ be the exradius of $T_i$ that touches a side of $\triangle ABC$.
Then
$$r_1r_3r_5R_1R_3R_5=r_2r_4r_6R_2R_4R_6.$$
\end{theorem}

\begin{figure}[h!t]
\centering
\includegraphics[width=0.5\linewidth]{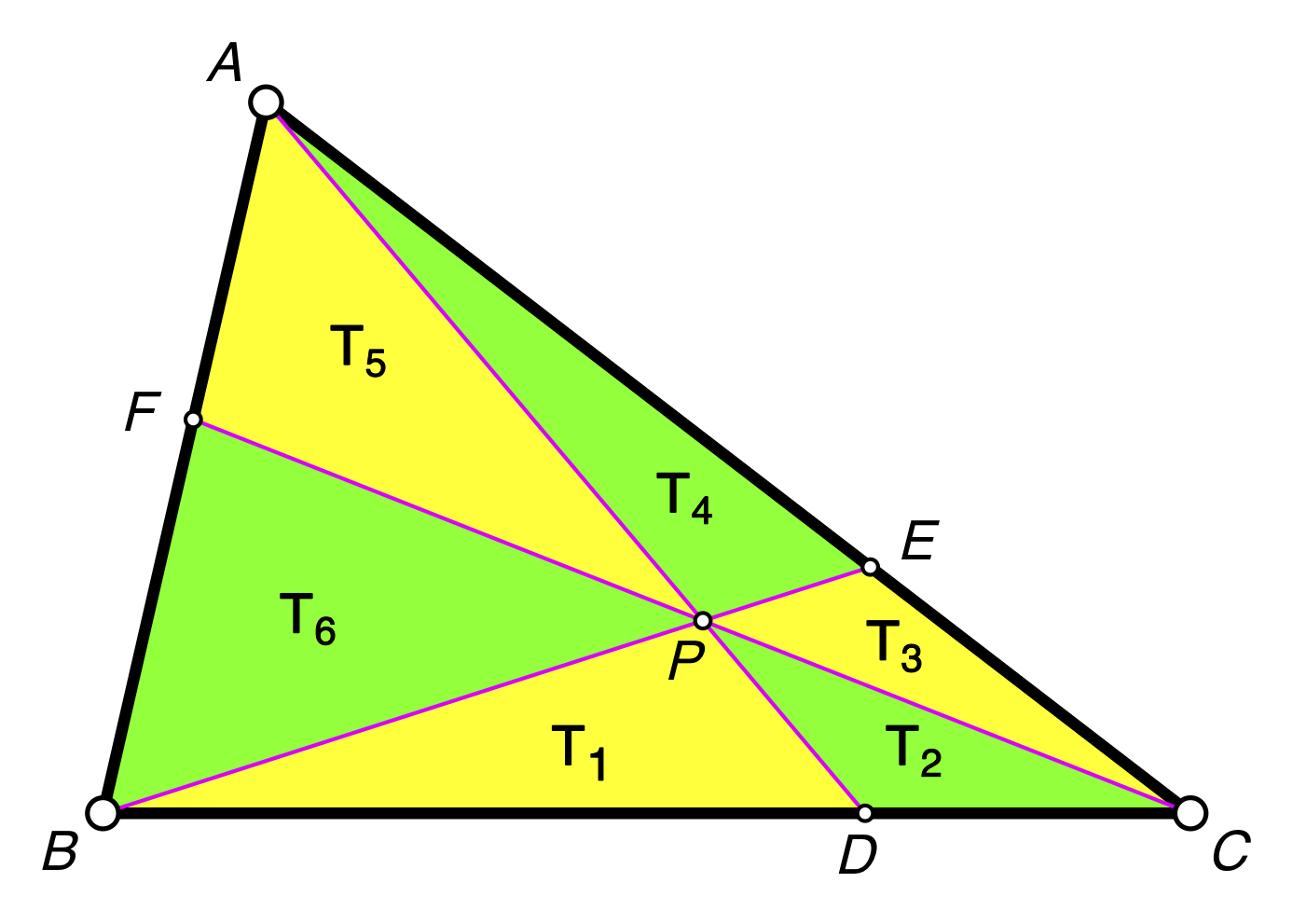}
\caption{six triangles}
\label{fig:sixTriangles}
\end{figure}

\begin{proof}
Note that in Figure~\ref{fig:sixTriangles}, $AD$ and $BE$ are straight lines passing through $P$, so $\angle BPD=\angle APE$.
By Lemma \ref{lemma:angle-invariant}, $r_1R_1/K_1=r_4R_4/K_4$, with similar identities
for the other two pairs of triangles. Therefore,
$$\frac{r_1R_1}{K_1}\cdot \frac{r_3R_3}{K_3}\cdot \frac{r_5R_5}{K_5}=
\frac{r_4R_4}{K_4}\cdot \frac{r_3R_6}{K_6}\cdot \frac{r_2R_2}{K_2}.$$
But $K_1K_3K_5=K_2K_4K_6$ by Theorem 7.4 from \cite{Rabinowitz}. Thus,
$r_1r_3r_5R_1R_3R_5=r_2r_4r_6R_2R_4R_6.$
\end{proof}

\textbf{Alternate formulation of Theorem \ref{thm:in-ex6}:}
\textit{In Figure \ref{fig:twelveCircles},
the product of the radii of the yellow circles is equal to the product of the radii of the green circles.}

\begin{figure}[h!b]
\centering
\includegraphics[width=0.5\linewidth]{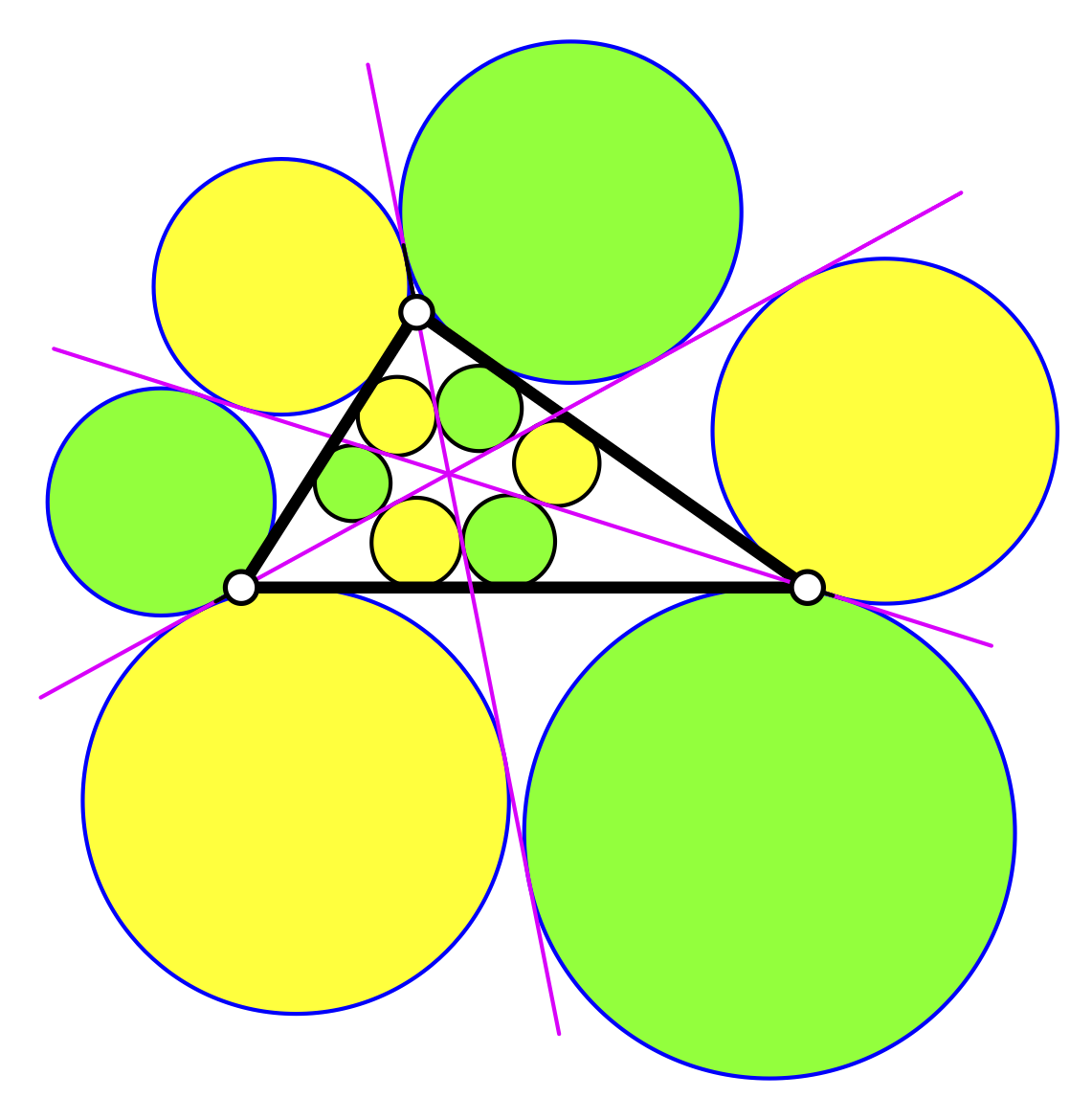}
\caption{twelve circles}
\label{fig:twelveCircles}
\end{figure}

Theorem \ref{thm:in-ex6} provides an alternate proof to some earlier theorems.
Applying Theorem \ref{thm:in-ex6} to Theorem \ref{thm:Orthocenter1} yields Theorem \ref{thm:Orthocenter2}.
Applying Theorem \ref{thm:in-ex6} to Theorem \ref{thm:Nagel1} yields Theorem \ref{thm:Nagel2}.

\goodbreak
We also have the following companion result, coloring the circles differently.

\textbf{Consequence of Theorem \ref{thm:in-ex}:}
\textit{In Figure \ref{fig:twelveCircles2}, the sum of the reciprocals of the radii of the yellow circles is equal to the sum of the reciprocals of the radii of the green circles.}

\begin{figure}[h!t]
\centering
\includegraphics[width=0.5\linewidth]{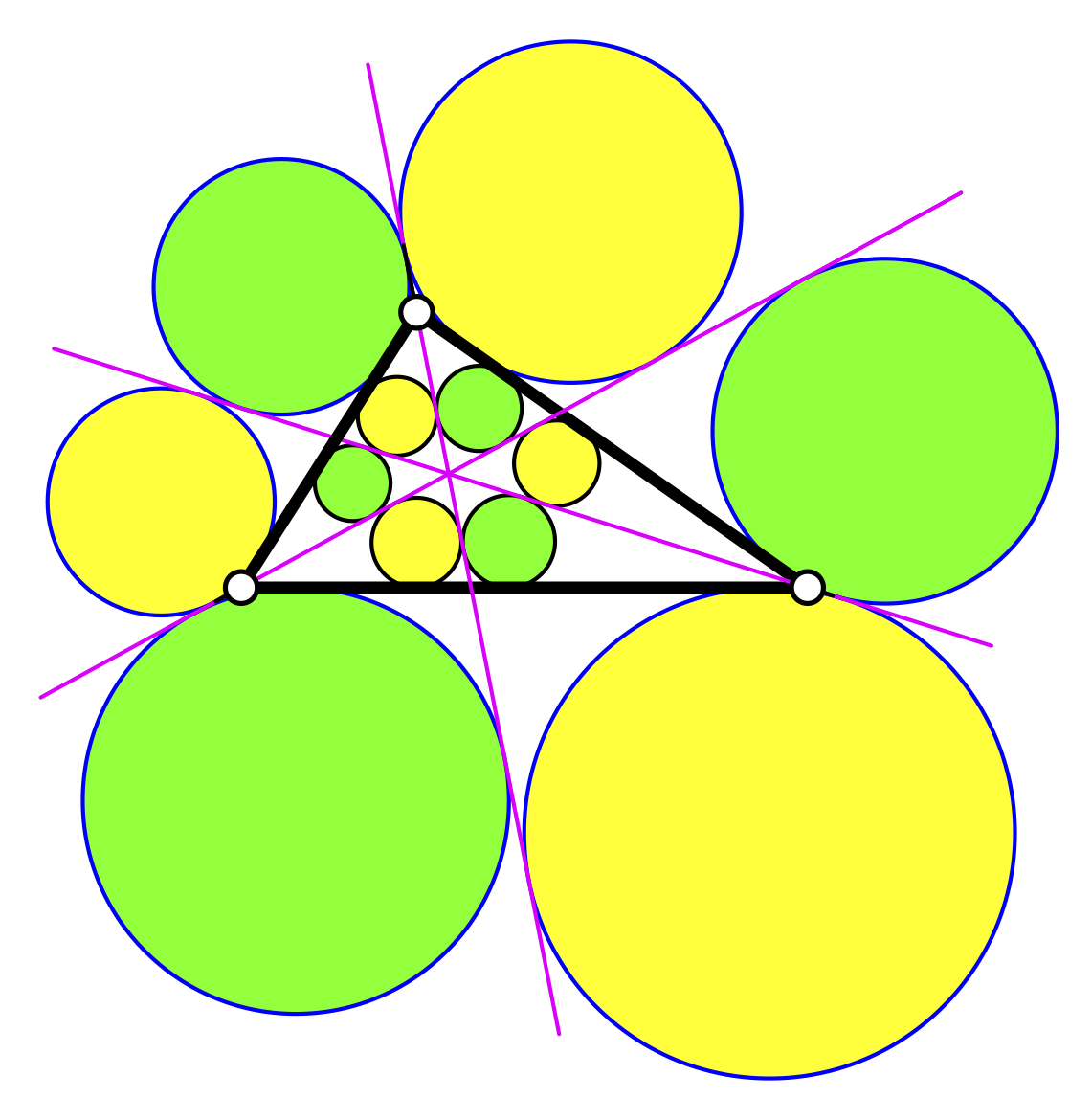}
\caption{twelve circles}
\label{fig:twelveCircles2}
\end{figure}

\section{The Circumcenter}

The following theorem involving incircles was proven in \cite{Rabinowitz}.

\begin{theorem}
\label{thm:in-circumcenter}
Let $O$ be the circumcenter of $\triangle ABC$. The cevians through $O$ divide $\triangle ABC$
into six small triangles and circles are inscribed in these triangles as shown in Figure \ref{fig:in-circumcenter}.
The circle labeled $i$ in the figure has radius $r_i$.
Then
$$\frac{1}{r_1}+\frac{1}{r_3}+\frac{1}{r_5}=\frac{1}{r_2}+\frac{1}{r_4}+\frac{1}{r_6}.$$
\end{theorem}

\begin{figure}[h!t]
\centering
\includegraphics[width=0.38\linewidth]{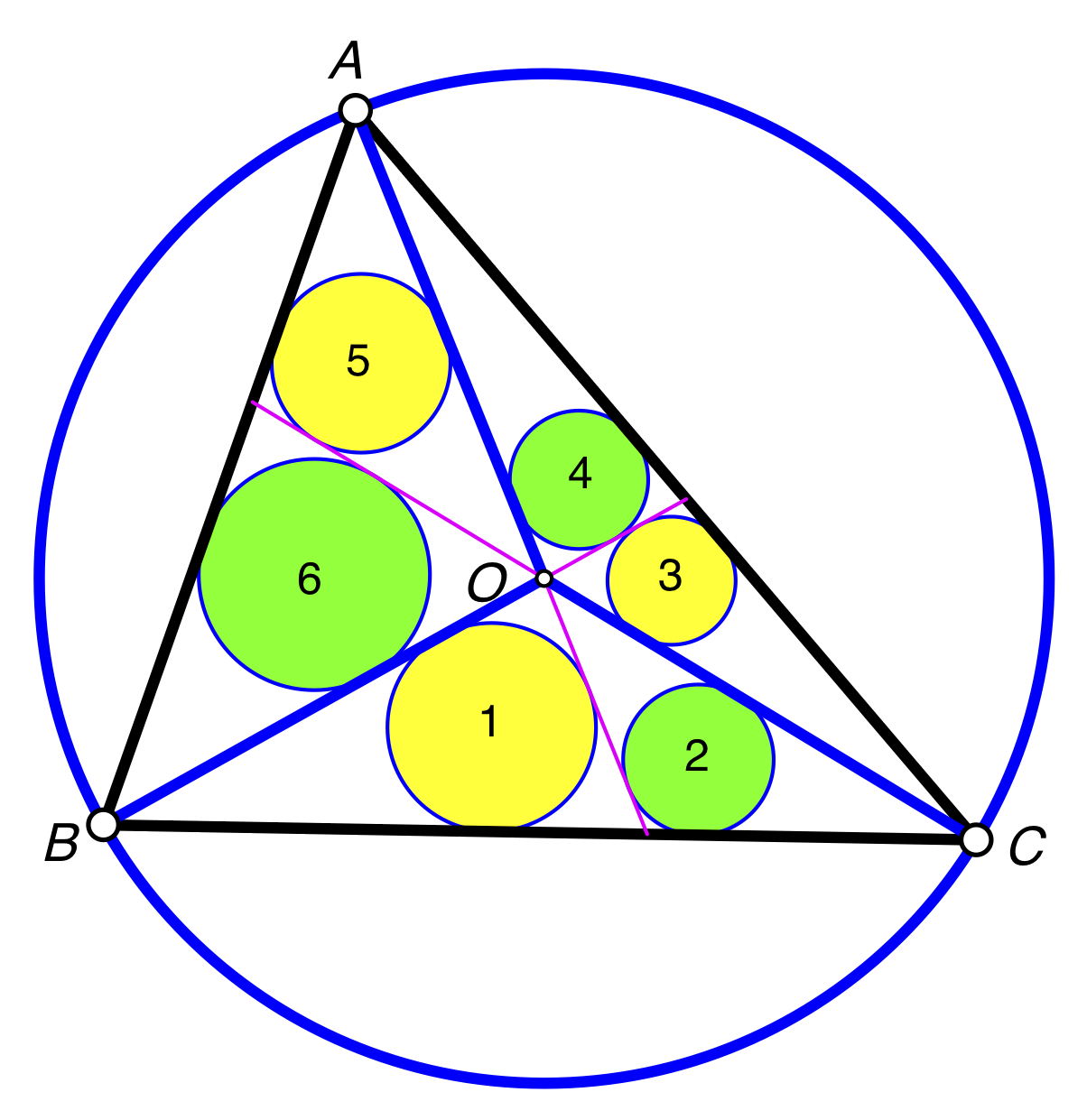}
\caption{circumcenter with incircles}
\label{fig:in-circumcenter}
\end{figure}

\goodbreak
Applying Theorem \ref{thm:in-ex} yields the following result about excircles.

\begin{theorem}
\label{thm:circumcenter}
Let $O$ be the circumcenter of $\triangle ABC$ and let the cevians through $O$ be $AD$, $BE$, and $CF$.
Let $r_1$ through $r_6$ be the radii of six circles tangent to the sides of $\triangle ABC$ and the cevians through $O$ situated as shown in Figure~\ref{fig:circumcenter}.
Then
$$\frac{1}{r_1}+\frac{1}{r_3}+\frac{1}{r_5}=\frac{1}{r_2}+\frac{1}{r_4}+\frac{1}{r_6}.$$
\end{theorem}

\begin{figure}[h!t]
\centering
\includegraphics[width=0.38\linewidth]{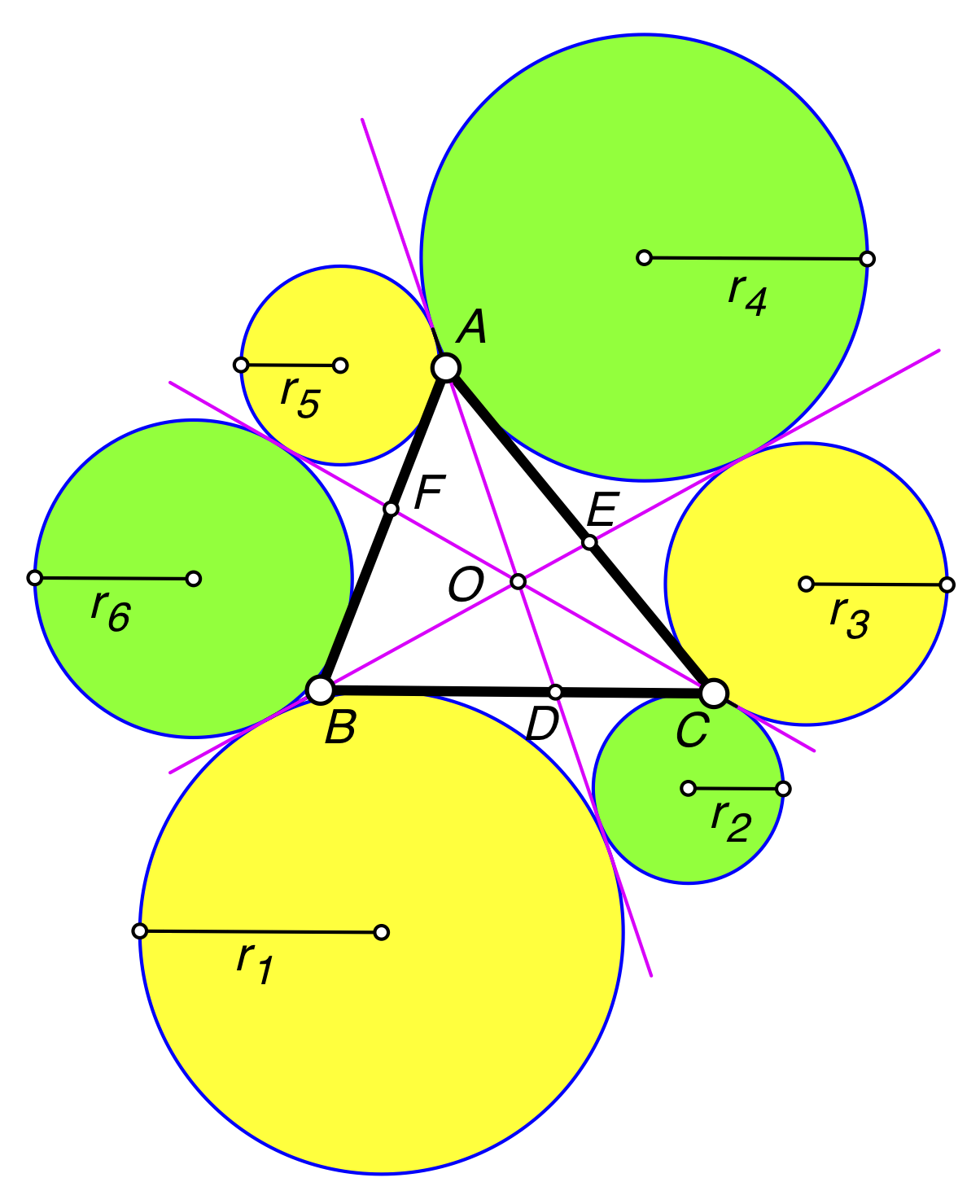}
\caption{circumcenter with excircles}
\label{fig:circumcenter}
\end{figure}

\bigskip
\section{The Incenter}

The following theorem involving incircles was proven in \cite{Rabinowitz}.

\begin{theorem}
\label{thm:in-incenter60}
Let $I$ be the incenter of $\triangle ABC$ and suppose $\angle ABC=60\degrees$. The cevians through $I$ divide $\triangle ABC$
into six small triangles and circles are inscribed in these triangles as shown in Figure \ref{fig:in-incenter60}.
The circle labeled $i$ in the figure has radius $r_i$.
Then
$$\frac{1}{r_1}+\frac{1}{r_4}+\frac{1}{r_5}=\frac{1}{r_2}+\frac{1}{r_3}+\frac{1}{r_6}.$$
\end{theorem}

\begin{figure}[h!t]
\centering
\includegraphics[width=0.6\linewidth]{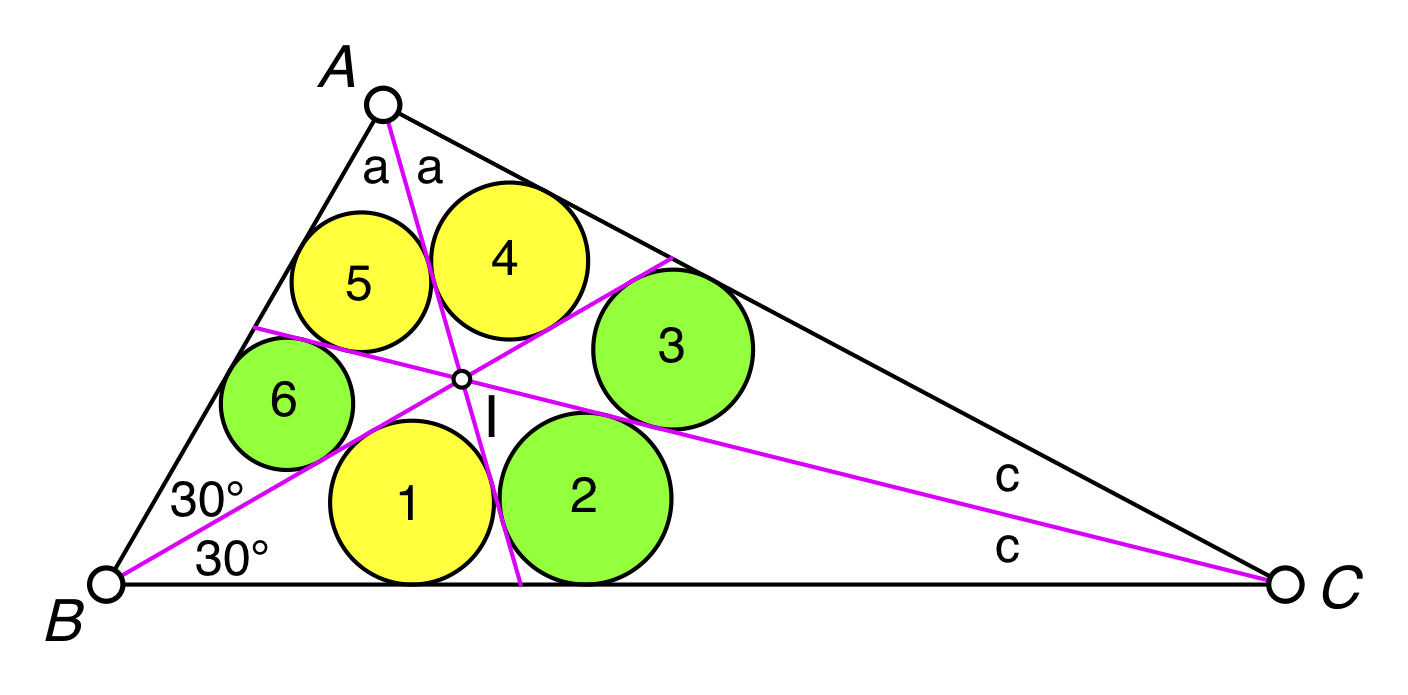}
\caption{incenter with incircles}
\label{fig:in-incenter60}
\end{figure}

\goodbreak
Applying Theorem \ref{thm:in-ex} yields the following result about excircles.

\begin{theorem}
\label{thm:incenter60}
Let $I$ be the incenter of $\triangle ABC$ and suppose $\angle ABC=60\degrees$.
Let $r_1$ through $r_6$ be the radii of six circles tangent to the sides of $\triangle ABC$ and the cevians through $I$ situated as shown in Figure~\ref{fig:incenter60}.
Then
$$\frac{1}{r_1}+\frac{1}{r_4}+\frac{1}{r_5}=\frac{1}{r_2}+\frac{1}{r_3}+\frac{1}{r_6}.$$
\end{theorem}

\begin{figure}[h!t]
\centering
\includegraphics[width=0.5\linewidth]{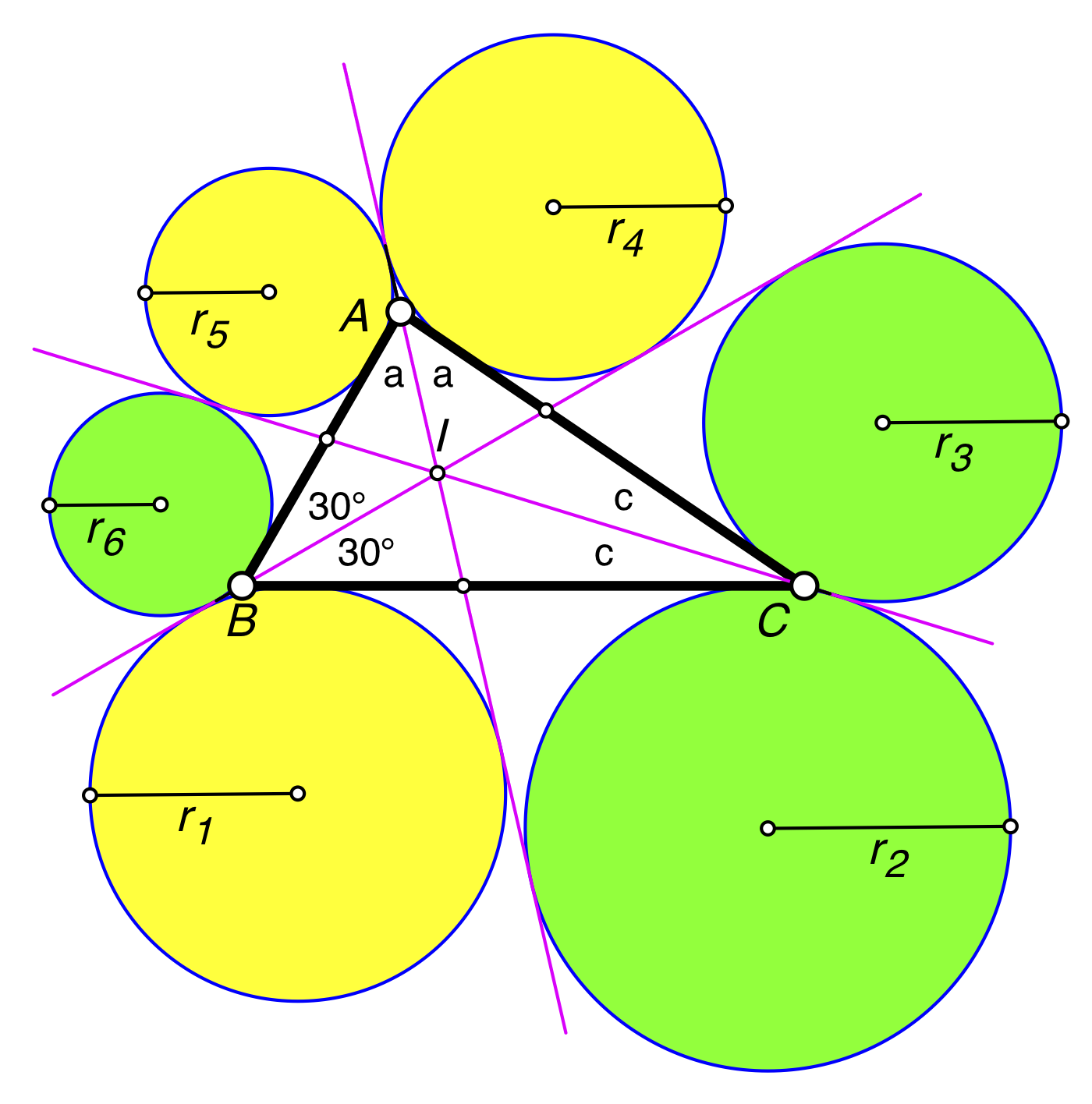}
\caption{incenter with excircles}
\label{fig:incenter60}
\end{figure}

When we change $\angle ABC$ from $60\degrees$ to $120\degrees$, we get a surprising result.

\begin{theorem}
\label{thm:incenter120}
Let $I$ be the incenter of $\triangle ABC$ and suppose $\angle ABC=120\degrees$.
Let $r_1$ through $r_6$ be the radii of six circles tangent to the sides of $\triangle ABC$ and the cevians through $I$ situated as shown in Figure~\ref{fig:figure120}.
Then
$$\frac{1}{r_1}+\frac{1}{r_3}+\frac{1}{r_4}+\frac{1}{r_6}=\frac{1}{r_2}+\frac{1}{r_5}.$$
\end{theorem}

\begin{figure}[h!t]
\centering
\includegraphics[width=0.5\linewidth]{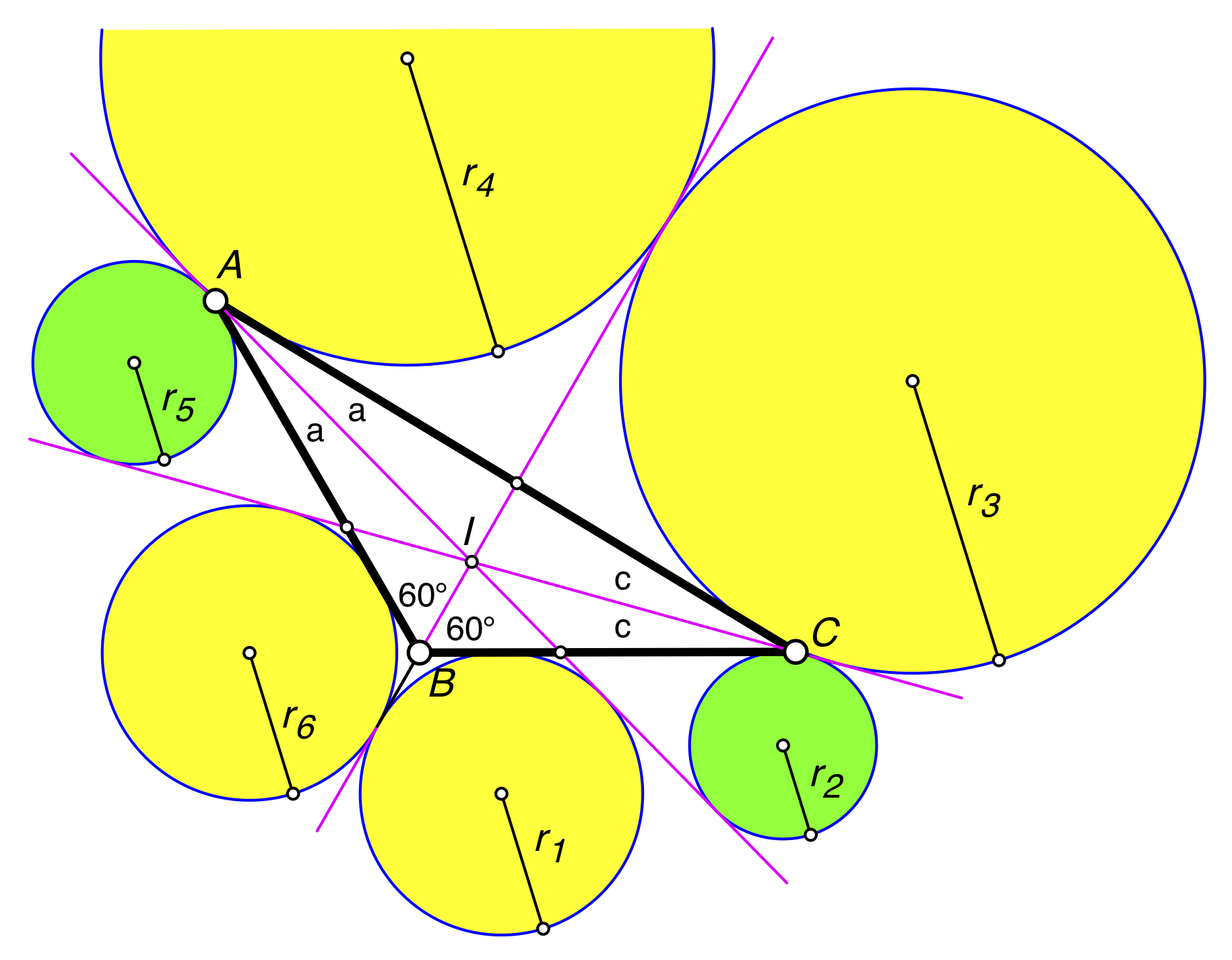}
\caption{incenter with excircles}
\label{fig:figure120}
\end{figure}

\begin{proof}
The proof is similar to the proof of Theorem 6.1 from \cite{Rabinowitz}.
Without loss of generality, we assume the circumradius of $\triangle ABC$ is $1/2$.
Then we use the Law of Sines to find the lengths of all the line segments associated with
$\triangle ABC$ and the three cevians in terms of the angles $a$ and $c$.
These values are given in \cite{Rabinowitz}. Then noting that $c=30\degrees-a$,
we use these values to compute the values of the $r_i$. Plugging these values into
the expression
$$\frac{1}{r_1}+\frac{1}{r_3}+\frac{1}{r_4}+\frac{1}{r_6}-\frac{1}{r_2}-\frac{1}{r_5}$$
and simplifying (using a computer algebra system) gives 0.
\end{proof}

The proof of Theorem \ref{thm:incenter60} depends on the proof of Theorem \ref{thm:in-incenter60}
whose only known proof (from \cite{Rabinowitz}) involves computer simplification of
complicated trigonometric expressions.

\begin{open}
Are there simple proofs of Theorems \ref{thm:incenter60} and \ref{thm:incenter120}
similar in complexity to the other proofs in this paper?
\end{open}
In Theorems \ref{thm:incenter60} and \ref{thm:incenter120}, $\angle ABC$ has a fixed value ($60\degrees$ or $120\degrees$). We can wonder what happens if we drop this restriction.

\begin{theorem}
\label{thm:ex-incenter-relation}
Let $I$ be the incenter of $\triangle ABC$.
Let $r_1$ through $r_6$ be the radii of six circles tangent to the sides of $\triangle ABC$ and the cevians through $I$ situated as shown in Figure~\ref{fig:ex-incenter-relation}.
If $\alpha=\frac{1}{2}\angle BAC$, $\beta=\frac{1}{2}\angle CBA$, $\gamma=\frac{1}{2}\angle ACB$,
then
$$\frac{\cos\gamma}{r_1}+\frac{\cos\alpha}{r_3}+\frac{\cos\beta}{r_5}=
\frac{\cos\beta}{r_2}+\frac{\cos\gamma}{r_4}+\frac{\cos\alpha}{r_6}.$$
\end{theorem}

\begin{figure}[h!t]
\centering
\includegraphics[width=0.5\linewidth]{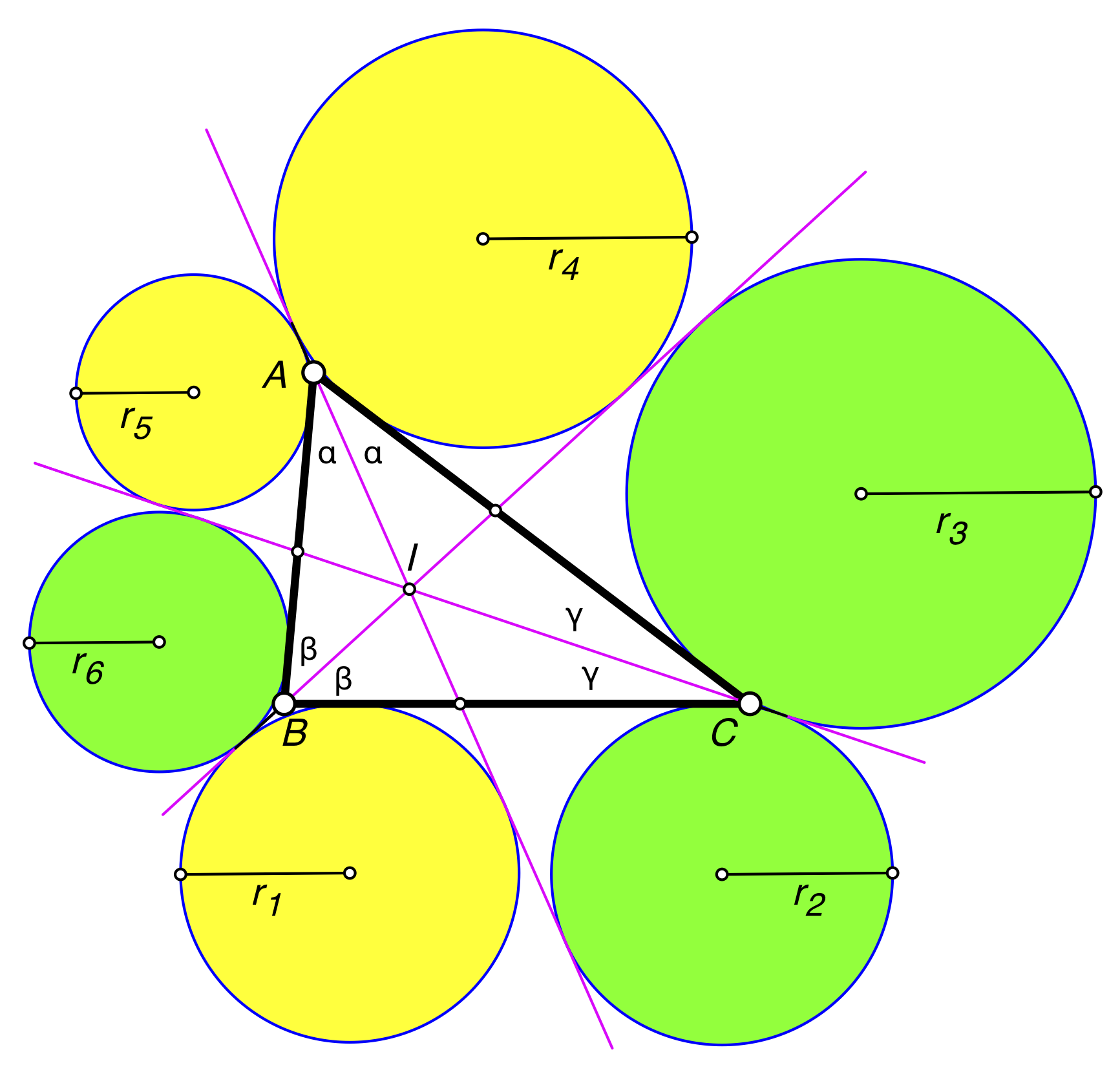}
\caption{incenter with excircles}
\label{fig:ex-incenter-relation}
\end{figure}

\begin{proof}
As with Theorem \ref{thm:incenter120}, the proof uses computer simplification of
trigonometric expressions for the lengths of the segments in the figure.
\end{proof}

There can be many relationships between the $r_i$. Although Theorem \ref{thm:ex-incenter-relation} involves $\alpha$, $\beta$, and $\gamma$, this does not preclude the existence of a relationship involving only the $r_i$.

\begin{open}
Is there a simple relationship between the radii of excircles associated with an arbitrary triangle
and the cevians through its incenter that does not depend on the shape of the triangle?
\end{open}



\bigskip
\begin{thebibliography}{9}

\bibitem{Altshiller-Court}
Nathan Altshiller-Court, \textit{College Geometry}, 2nd edition, Dover Publications, New York, 2007.

\bibitem{Duca}
Gheorghe Duca, \textit{Solution to Problem 3383}. Romantics of Geometry, August, 2019.\\
\url{https://www.facebook.com/groups/parmenides52/permalink/2364455557001470/}

\bibitem{Fukagawa}
Hidetoshi Fukagawa, \textit{Traditional Japanese Mathematics Problems of the 18th \& 19th Centuries}, SCT Publishing, Singapore, 2002.

\bibitem{Krishna}
Dasari Naga vijay Krishna, \textit{The fundamental property of Nagel point--A new proof}. Journal of Science and Arts, \textbf{38}(2017)31--36.\\
\url{https://www.researchgate.net/publication/315725322_THE_FUNDAMENTAL_PROPERTY_OF_NAGEL_POINT_-_A_NEW_PROOF}

\bibitem{Rabinowitz}
Stanley Rabinowitz, \textit{Relationships Between Six Incircles}. Sangaku Journal of Mathematics, \textbf{3}(2019)51-66.
\url{http://www.sangaku-journal.eu/2019/SJM_2019_51-66_Rabinowitz.pdf}

\bibitem{Ushijima}
Seiyou Ushijima, \textit{Zoku Sangaku Shousen}, 1823.

\end{thebibliography}
\end{document}